\documentclass[12pt]{article}

\usepackage{fullpage}

\usepackage{amsmath}
\usepackage{amssymb}
\usepackage{amsthm}
\usepackage{bibunits}
\usepackage{bm}
\usepackage{caption}
\usepackage{color}
\usepackage{fullpage}
\usepackage{graphicx}
\usepackage{hyperref}
\usepackage{natbib}
\usepackage{placeins}
\usepackage{subcaption}
\usepackage{cleveref} 

\newtheorem{proposition}{Proposition}
\newtheorem{theorem}{Theorem}
\newtheorem{definition}{Definition}
\newtheorem{corollary}{Corollary}
\newtheorem{lemma}{Lemma}
\newtheorem{assumption}{Standing Assumption}

\newtheorem{example}{Example}
\newtheorem{desiderata}{Desiderata}
\newtheorem{kernel}{Kernel}

\crefname{assumption}{Assumption}{Assumptions}
\crefname{figure}{Figure}{Figures}
\crefname{lemma}{Lemma}{Lemmas}

\DeclareMathOperator*{\arginf}{arg\,inf}

\begin{document}

\title{A Riemann--Stein Kernel Method}

\author{Alessandro Barp$^{1,2,*}$, Chris. J. Oates$^{3,2,*}$, Emilio Porcu$^{4,5}$, Mark Girolami$^{1,2}$ \vspace{5pt} \\
\small $^1$University of Cambridge, UK \\
\small $^2$Alan Turing Institute, London, UK \\ 
\small $^3$Newcastle University, UK \\ 
\small $^4$Khalifa University, Abu Dhabi \\
\small $^5$Trinity College Dublin, IE \\ 
\small $^*$To whom correspondence should be addressed.}

\maketitle

\begin{abstract}
This paper proposes and studies a numerical method for approximation of posterior expectations based on interpolation with a Stein reproducing kernel.
Finite-sample-size bounds on the approximation error are established for posterior distributions supported on a compact Riemannian manifold, and we relate these to a \emph{kernel Stein discrepancy} (KSD).
Moreover, we prove in our setting that the KSD is equivalent to \emph{Sobolev discrepancy} and, in doing so, we completely characterise the convergence-determining properties of KSD.
Our contribution is rooted in a novel combination of Stein's method, the theory of reproducing kernels, and existence and regularity results for partial differential equations on a Riemannian manifold.

\end{abstract}

\begin{bibunit}[plainnat]

\section{Introduction} \label{sec: intro}

The focus of this paper is the numerical approximation of an intergal
\begin{eqnarray}
\int_M f \; \mathrm{d}P \label{eq: integral}
\end{eqnarray}
where $P$ is a probability measure on a compact,  connected, complete Riemannian manifold $M$ and $f \in C(M,\mathbb{R})$.
The notation $C(A,B)$ indicates the set of continuous functions from a topological space $A$ to a topological space $B$.
The distribution $P$ is assumed to admit a density $p$ with respect to the natural volume measure on $M$,  specified only up to an unknown normalisation constant.
It is assumed that direct computation of the normalisation constant is difficult and therefore precluded.
This final point demands special consideration and prevents standard numerical integration method from being used.
This situation is of course regularly encountered in Bayesian statistics, where $P$ is a posterior distribution whose density $p$ is specified in un-normalised form as the product of a prior and a likelihood.
In general the direct computation of the normalisation constant is difficult in the Bayesian context \citep{Gelman1998}.

Several approaches to approximation of \eqref{eq: integral} have been developed.
These range from heuristic approaches, such as variational inference \citep{Kingma2014} and the Laplace approximation \citep{Rue2009}, through to asymptotically exact approaches such as Markov chain Monte Carlo \citep[MCMC;][]{Gilks1995}.
Among asymptocially exact approaches, MCMC is most widely-used and its convergence theory is well-developed \citep{Meyn2012}.
However, the absolute error of the ergodic average is gated at $O_p(n^{-1/2})$, where $n$ is the number of evaluations of the integrand.
This rate is sub-optimal for an $s$-times weakly differentiable integrand when $s > \frac{d}{2}$ and $d$ is the dimension of the manifold; a consequence of the fact that the ergodic average does not exploit smoothness properties of the integrand \citep{Traub2003}.
In recent years, several alternatives to MCMC have been developed to address this convergence bottleneck, focussing on smooth integrals of low effective dimension for which sub-optimality of MCMC is most pronounced.
These include transport maps \citep{Marzouk2016}, Riemann sums \citep{Philippe2001}, quasi Monte Carlo ratio estimators \citep{Schwab2012,Dick2016,Dick2019}, minimum energy designs \citep{Joseph2017}, support points \citep{Mak2016} and estimators based on Stein's method \citep{Liu2016,Oates2017,Liu2017b,Oates2018,Chen2018,Chen2019,Hodgkinson2020,Riabiz2020,Teymur2020,Fisher2020}.
The computational cost of some of these methods is higher than $O(n)$ and, for the method that we study in this work, the cost is $O(n^3)$.
Thus any accelerated convergence being offered must be weighed against this increased computational overhead.

\subsection{Context}

The purpose of this paper is to study the numerical approximation of posterior expectations based on interpolation with a Stein reproducing kernel, an approach first proposed in the Euclidean context in \cite{Oates2017}.
To this end, we recall how Stein's method \citep{Stein1972} can be used for the numerical approximation of \eqref{eq: integral}.
Let $\mathcal{P}(M)$ denote the space of Borel distributions on $M$.

\begin{definition}
Fix $P \in \mathcal{P}(M)$.
Consider a set $\mathcal{H} \subset C(M,\mathbb{R})$ and an operator $\tau : \mathcal{H} \rightarrow C(M,\mathbb{R})$ with the property that, for all $Q \in \mathcal{P}(M)$,
\begin{eqnarray*}
P \; = \; Q & \Leftrightarrow & \int_M \; \tau h \; \mathrm{d}Q \; = \; 0  \qquad \forall h \in \mathcal{H} . \label{def: Stein class}
\end{eqnarray*}
Then $(\mathcal{H} , \tau)$ is said to be a \emph{Stein characterisation} of $P$.
In this case the set $\mathcal{H}$ is called a \emph{Stein class} and the operator $\tau$ is called a \emph{Stein operator}.
If only the $\Rightarrow$ implication holds, so that $\mathcal{H}$ need not be rich enough to distinguish elements in $\mathcal{P}(M)$, then we call $(\mathcal{H},\tau)$ a \emph{Stein pair} for $P$. 
\end{definition}
\noindent The definition of a Stein characterisation is classical, but in this paper only the (novel) definition of a Stein pair will be used.
The reader is referred to \cite{Ley2017} for further background on Stein's method.
For the moment, let us suppose that a Stein pair $(\mathcal{H},\tau)$ for $P$ can be found.
Then \eqref{eq: integral} can be approximated in direct a manner, that will now be explained:
First, select a set of distinct locations $X = \{\bm{x}_i\}_{i=1}^n \subset M$ at which the integrand is to be evaluated.
Then construct an estimator of the form
\begin{eqnarray}
P_X(f) & := & \arginf_{\xi \in \mathbb{R}} \; \inf_{h \in \mathcal{H}} \; \sum_{i=1}^n \Big( \xi + \tau h(\bm{x}_i) - f(\bm{x}_i) \Big)^2 + R_1(\xi) + R_2(h)  \label{eq: Stein method}
\end{eqnarray}
where $R_1$ and $R_2$ are regularisation terms to be specified, whose purpose is to ensure that a (unique) minimum will exist.
The form of \eqref{eq: Stein method} can be motivated as constructing an approximation $f_X$ to the integrand $f$, based on evaluation of $f$ at the locations in $X$, in the class of functions of the form $f_X :=\xi + \tau h$, where $\xi \in \mathbb{R}$ and $h \in \mathcal{H}$.
The definition of a Stein pair ensures that $\int_M f_X \mathrm{d}P = \xi$, so \eqref{eq: Stein method} can be interpreted as the integral of an approximation $f_X$ to the integrand.
Numerical approximations of the form \eqref{eq: Stein method} are increasingly being adopted in applied and methodological work \citep{liu2017action,south2017efficient,brosse2018diffusion,roussel2019perturbative,lam2019stability,mijatovic2019asymptotic}.
The properties of \eqref{eq: Stein method} depend on the class $\mathcal{H}$, the operator $\tau$, the set of points $X$ where $f$ is evaluated, and the regularisation terms $R_1,R_2$.
Thus the above formulation is quite general and some specific choices are discussed next.

\paragraph*{Previous Work for $M = \mathbb{R}^d$:}

All previous work on \eqref{eq: Stein method} has focussed on the Euclidean context with $M = \mathbb{R}^d$.
In particular, \cite{Assaraf1999,Mira2013} considered the case where $\mathcal{H} \subset C(\mathbb{R}^d, \mathbb{R})$ is a space of low-degree polynomials under no regularisation, i.e. $R_1, R_2 \equiv 0$.
This was combined with the operator $\tau : \mathcal{H} \rightarrow C(\mathbb{R}^d,\mathbb{R})$, $\tau h = \nabla \cdot (p \nabla h) / p$, a second-order differential operator that can be evaluated without access to $p$'s normalisation constant.
This led to an over-constrained least-squares problem and \eqref{eq: Stein method} can be seen as a classical control variate method.
Higher-order polynomials were more considered in \cite{South2018} in conjunction with an appropriate regularisation to ensure the least-squares problem remains well-posed.

The innovation in \cite{Oates2017} was to consider instead an infinite-dimensional normed space for $\mathcal{H}$.
The operator in \cite{Oates2017} was $\tau : \mathcal{H} \rightarrow C(\mathbb{R}^d,\mathbb{R})$, where $\mathcal{H} \subset C(\mathbb{R}^d,\mathbb{R}^d)$ was a Cartesian product of reproducing kernel Hilbert spaces and $\tau h = \nabla \cdot (p h) /p$ was a first-order differential operator that can again be evaluated without $p$'s normalisation constant.
To complete the specification of the method, the following natural regularisation terms were proposed:
\begin{align} 
\begin{aligned}
R_1(\xi) & \; = \; \sigma^{-2} \xi^2 \\
R_2(h) & \; = \; \inf\{ \|h'\|_{\mathcal{H}}^2 : h' \in \mathcal{H}, \; \tau (h' - h) = 0 \} \label{eq: Stein kernel method}
\end{aligned}
\end{align}
where $\sigma > 0$ was a parameter to be specified and $\|\cdot\|_{\mathcal{H}}$ denotes the norm associated to $\mathcal{H}$.
That the term $R_2(h)$ should depend on $h$ through $\tau h$ is natural, since $h$ enters into the approximation $f_X$ only through $\tau h$.
The solution to \eqref{eq: Stein method} can be computed in closed-form when reproducing kernels are employed and, since in this case $f_X$ interpolates the data $X$, we call $f_X$ a \emph{Stein kernel interpolant}.
In recent work \cite{South2020} proposed a semi-parametric approach that combined elements from \cite{Oates2017} and \cite{Assaraf1999,Mira2013}.

In parallel work in the Euclidean context, \cite{Liu2017b} considered adding additional regularisation to \eqref{eq: Stein method}, whilst alternatives to the squared error objective have also been considered, including an empirical variance estimator \citep{Belomestny2017} and an estimator of asymptotic variance \citep{belomestny2020variance} for use in the MCMC context.
The latter were shown, empirically, to improve estimator performance but at the cost of no longer having a closed-form expression for the estimator.
In a different direction, \cite{Zhu2018} proposed to take $\mathcal{H}$ to be a finite-dimensional parametric neural network and empirically explored its potential, while \cite{Si2020} provided a theoretical analysis of stochastic gradient descent in that context.

\paragraph*{Related Work for General $M$:}

Beyond the Euclidean case, in parallel to the present research, a Stein pair for distributions supported on a Riemannian manifold was provided in \cite{Liu2017,xu2020stein} and this was extended to a Stein characterisation in \cite{Le2020}.
The focus of \cite{Liu2017} was the design of a gradient flow on a manifold, the focus of \cite{xu2020stein} was on goodness-of-fit testing and the focus of \cite{Le2020} was weak convergence control.
Our analysis starts from the same Stein pair, but diverges thereafter, with our focus being on the numerical approximation of \eqref{eq: integral}.

\subsection{Our Contributions}

Our contributions are as follows:
\begin{itemize}
\item {\it Generalisation to a Riemannian Manifold:}
Integrals on manifolds arise in many important applications of Bayesian statistics, most notably directional statistics \citep{Mardia2000} and modelling of functional data on the sphere $\mathbb{S}^2$ \citep{Porcu2016}.
In this context, MCMC methods have been developed to sample from distributions defined on a manifold \citep[e.g.][]{Diaconis2013,barp2019hamiltonian,Byrne2013,Lan2014,Holbrook2016,arnaudon2019irreversible}.
In this paper we complement existing work by generalising the Stein kernel interpolant of \cite{Oates2017} to integrals defined on a Riemannian manifold.

\item {\it New Insight into Discrepancy:} The {\it kernel Stein discrepancy} \citep[KSD;][]{Chwialkowski2016,Liu2016b,Gorham2017} is an upper bound on the worst-case error of the Stein kernel interpolant.
For our setting, we prove that this discrepancy is equivalent to a classical Sobolev discrepancy on the manifold.
An explicit statement is presented in \Cref{thm: equivalent kernels}.
Thus we completely characterise the convergence-determining properties of KSD in this context.
This is our main result and our proof combines Stein's method, the theory of reproducing kernels and existence and regularity results for partial differential equations on a Riemannian manifold.

\item {\it Finite-Sample-Size Error Bound:} 
Bounds on the worst-case approximation error are established as consequences of the Sobolev connection in \Cref{thm: equivalent kernels}.
The optimal convergence rate is established, under appropriate regularity assumptions on the distribution $P$ and the point set $X$.
An explicit statement of the rate is presented in \Cref{thm: main result}.
\end{itemize}
The paper now proceeds to \Cref{sec: background}, where we provide a brief mathematical background.
In \Cref{subsec: statement} we state our main results, with discussion provided in \Cref{sec: discussion}.

\section{Background} \label{sec: background}

The purpose of this section is to introduce the mathematical tools that are needed for our development.
The experienced reader may prefer to continue directly to \Cref{subsec: statement} and return later to \Cref{sec: background} if required.
In \Cref{subsec: RM} we recall the definition of a Riemannian manifold, in \Cref{subsec: GMT} we review geometric measure theory, in \Cref{subsec: CRM} we describe calculus on a Riemannian manifold, in \Cref{subsec: RKHS} we define reproducing kernel Hilbert spaces and in \Cref{subsubsec: sobolev norm} we construct a Sobolev norm on a Riemannian manifold.

\subsection{Riemannian Manifolds} \label{subsec: RM}

Let $C^l(A,B)$ denote the set of all $l$-times continuously differentiable functions from $A$ to $B$, where $A$ and $B$ posess suitable structure for the notion of a continuous derivative to be well-defined.

A $d$-dimensional \emph{manifold} $M$, $d \in \mathbb{N}$, is a Hausdorff topological space for which every point $\bm{x} \in M$ has an open neighbourhood $U_{\bm{x}}$ homeomorphic to an open subset of $\mathbb R^d$ or an open subset of the closed $d$-dimensional upper half-space if $M$ has a boundary $\partial M$ \citep[see p25 of][]{Lee2013}. 
If $\phi:U \rightarrow \mathbb R^d$ is a homemorphism (onto its image) with $\bm{x} \in U \subset M$, we say $(U,\phi)$ is a \emph{coordinate patch} around $\bm{x}$. 
This defines \textit{coordinate functions} $q_j :=\pi_j \circ \phi$  over $U$, where $\pi_j:\mathbb R^d \rightarrow \mathbb R$ are the canonical projections and $\circ$ denotes composition of functions. 
A $C^\infty$ \emph{atlas} is a collection of coordinate patches $(U_i, \phi_i)$ that cover $M$ such that the \textit{transition functions} $\phi_j \circ \phi_i^{-1}$ are $C^\infty( \mathbb R^d, \mathbb{R}^d)$ whenever they are defined. 
A manifold is \emph{smooth} if it posesses a $C^\infty$ atlas.
The \emph{tangent space} $T_{\bm{x}}M$ at $\bm{x}$ is the vector space of linear functionals over $C^{\infty}(M)$ satisfying Leibniz rule. 
If $q_j$ are coordinates on a patch $(U, \phi)$ containing $\bm x$, the \textit{coordinate vectors} $\partial_{q_j} |_{\bm x}:f \mapsto   \partial (f \circ \phi_j^{-1}) / \partial q_j |_{\phi( \bm x ) }$ define a basis of  $T_{\bm{x}}M$.
We say $(M,g)$ is a \emph{Riemannian} manifold if $M$ has a \textit{metric tensor} $g$, i.e., a smooth map $\bm{x} \mapsto g_{\bm{x}}$ such that $g_{\bm{x}}$ is an inner product on $T_{\bm{x}}M$. 
It will be convenient to represent the metric tensor $g_{\bm{x}}$ as a matrix $\mathrm{G}(\bm{x})$ with coordinates  $\mathrm{G}_{ij}(\bm{x}):= g_{\bm{x}}\big( \partial_{q_i} |_{\bm x},\partial_{q_j} |_{\bm x} \big)  $.

The non-Euclidean manifolds of interest in most applications are closed; for example the group of rotations, Grassmannian manifolds, and Stiefel manifolds such as hyperspheres, though we will also allow manifolds with boundary for completeness.

\begin{assumption}
$M$ is a smooth, compact, complete, and connected manifold, that is either closed or is a manifold with boundary.
\end{assumption}

In the case of a manifold with boundary, the outward-pointing \textit{unit normal} $\bm{n}$ to the boundary $\partial M$ of the manifold can be defined via the fact that, since $(\partial M,\left. g \right|_{\partial M})$ is a Riemannian submanifold of $(M,g)$, then for each $\bm{x} \in \partial M$, the metric $g_{\bm{x}}$ of $M$ splits the tangent space $T_{\bm{x}}M$ into $T_{\bm{x}} \partial M$ and its orthogonal complement $N_{\bm{x}}$; i.e. $T_{\bm{x}} M = T_{\bm{x}} \partial M \oplus N_{\bm{x}}$. 
Here $N_{\bm{x}}$ is the one-dimensional space spanned by the normal vector to $\partial M$.
See e.g. \cite{Bachman2006}.

The \emph{geodesic distance} $d_M(\bm{x},\bm{y})$ on a complete Riemannian manifold is defined as the infimum of the length $ \int_0^1  \sqrt{g(\dot \gamma(t),\dot \gamma(t))}\mathrm{d}t$ over all continuously differentiable curves $\gamma:[0,1] \to M$ with $\gamma(0) = \bm{x}$ and $\gamma(1) = \bm{y}$, where $\dot \gamma$ is the tangent vector to $\gamma$.

\begin{example} \label{example: S2}
The sphere $\mathbb{S}^2$ is a 2-dimensional Riemannian manifold.
The coordinate patch $\phi$ with $\phi^{-1}(\bm{q})=(\cos q_1 \sin q_2 , \sin q_1 \sin q_2 , \cos q_2)$, with local coordinates $q_1 \in (0,2\pi)$, $q_2 \in (0,\pi)$, covers almost all $\bm{x} \in \mathbb{S}^2$ (it does not cover the half great circle that passes through both poles and the point $(1,0,0)$). 
The tangent space $T_{\phi^{-1}(q_1,q_2)}M$ is spanned by $\partial_{q_1}  = (-\sin q_1 \sin q_2, \cos q_1 \sin q_2, 0)$ and $\partial_{q_2} = (\cos q_1 \cos q_2, \sin q_1 \cos q_2, -\sin q_2)$. 
Let $s = s_1 \partial_{q_1} + s_2 \partial_{q_2} \in T_{\bm{x}} M$ be associated with the coefficient vector $\bm{s}^\top = [s_1 , s_2]$ and similarly for $\bm{t}^\top = [t_1, t_2]$.
Taking the Euclidean inner product of these vectors shows that $g_{\bm{x}}( s , t) := \langle \bm{s} ,\bm{t} \rangle_{\mathrm{G}} =\bm{s}^\top \mathrm{G}\bm{t}$ where $\mathrm{G}_{1,1}= \sin^2 q_2$, $\mathrm{G}_{2,2} = 1$, $\mathrm{G}_{1,2} = \mathrm{G}_{2,1} = 0$.
\end{example}

\subsection{Geometric Measure Theory} \label{subsec: GMT}

Any Riemannian manifold has a natural measure $V$ over its Borel algebra, called the \emph{Riemannian volume measure}, with infinitesimal volume element denoted $\mathrm{d}V$. 
In a coordinate patch $U_i \subset \mathbb{R}^d$, this measure can be expressed in terms of the Lebesgue measure: $\mathrm{d}V =\sqrt{\text{det}(\mathrm{G}(\bm{x}))} \lambda^d(\mathrm{d}\bm{q})$ where $V(f) = \int_M f \mathrm{d}V$. 
In particular, when $M$ is the Euclidean space this is just the Lebesgue measure, and when $M$ is an embedded manifold in $\mathbb R^m$, $V$ is the Hausdorff measure \citep{Federer1969}. 
When $M$ has a non-empty boundary $\partial M$, the latter  inherits a  measure 
by  noting  that $\partial M$ is a submanifold of $M$ and the restriction $\left. g \right|_{\partial M}$ of the metric $g$ turns it into a Riemannian manifold $(\partial M , \left. g \right|_{\partial M})$.
Its associated Riemannian volume measure is usually denoted  $i_{\bm{n}} \mathrm{d}V$ in the literature \cite{frankel2011geometry}.

\begin{example}
For the sphere $\mathbb{S}^2$, $\mathrm{d}V = \sin q_2 \mathrm{d} q_1 \mathrm{d} q_2$, where $\sin q_2$ is the area of the parallelogram spanned by $\partial_{q_1}, \partial_{q_2}$.
\end{example}

\subsection{Calculus on a Riemannian Manifold} \label{subsec: CRM}

To present a geometric, coordinate-independent construction of differential operators on manifolds would require either exterior calculus or the concept of a covariant derivative. 
To limit scope, we present two important differential operators in local coordinates and merely comment that the associated operators are in fact coordinate-independent; full details can be found in \cite{Bachman2006}. 
To this end, denote the gradient of a function $f : M \rightarrow \mathbb{R}$, assumed to exist, as
$$
\nabla f  \; = \; \sum_{i,j = 1}^d [\mathrm{G}^{-1}]_{ij} \frac{\partial f \circ \phi_j^{-1}}{\partial q_j} \partial_{q_i} .
$$
Likewise, define the divergence of a vector field $\bm{s} = s_1 \partial_{q_1} + \dots + s_d  \partial_{q_d}$ with $s_i = s_i(\bm{x})$, assumed to exist, as
$$
\nabla \cdot \bm{s} \; =\; \sum_{i=1}^d \frac{\partial}{\partial q_i} s_i + s_i \frac{\partial}{\partial q_i} \log \sqrt{\text{det}(\mathrm{G})}.
$$
These two differential operators are sufficient for our work; for instance, they can be combined to obtain the Laplace-Beltrami operator $\Delta f := \nabla \cdot \nabla f$.

\subsection{Reproducing Kernel Hilbert Spaces} \label{subsec: RKHS}

A reproducing kernel Hilbert space (RKHS) $\mathcal H$ of functions on $M$ is a Hilbert space for which the evaluation functionals $\delta_{\bm{x}}: \mathcal H \rightarrow \mathbb R$, $\delta_{\bm{x}}(f):= f(\bm{x})$, are continuous for each $\bm{x} \in M$. 
The dual space of bounded linear operators $L : \mathcal{H} \rightarrow \mathbb{R}$ is denoted $\mathcal{H}^*$.
The Riesz-representation theorem implies that $\mathcal{H}$ is isomorphic to $\mathcal H^*$ and we can thus associate a vector $k_{\bm{x}} \in \mathcal H$ to $\delta_{\bm{x}}$ which satisfies $f(\bm{x})= \langle f, k_{\bm{x}} \rangle_{\mathcal{H}}$. The symmetric function $k(\bm{x},\bm{y}):= k_{\bm{x}}(\bm{y})$ is called the \emph{reproducing kernel} for $\mathcal H$ and we denote this as $\mathcal{H}(k)$ in the sequel. 
It can be checked $k$ is a semi-positive definite function on $M \times M$. Moore's theorem states the converse is also true; any semi-positive function on $M$ defines an RKHS $\mathcal{H}(k)$ of functions on $M$ with $k$ as its reproducing kernel.
See e.g. \cite{Berlinet2011}.

\subsection{Sobolev Norm on a Riemannian Manifold} \label{subsubsec: sobolev norm}

Let $S \subset \mathbb R^d$ be an open set and recall that the standard \emph{Sobolev space} $W_2^s(S)$ is defined as the set of equivalence classes $f \in L^2(S)$ such that the weak derivatives $\mathrm{D}^\alpha f := \partial_{x_1}^{\alpha_1} \dots \partial_{x_d}^{\alpha_d} f \in L^2(S)$ for all $|\alpha| := \alpha_1 + \dots + \alpha_d \leq s$. 
The set $W_2^s(S)$, for $s > d/2$, becomes a RKHS when equipped with the norm\footnote{In a slight abuse of notation, for $f \in C(S, \mathbb{R})$ we interpret $\|f\|_{W_2^s(S)}$ as $\|\iota \circ f\|_{W_2^s(S)}$ where $\iota \circ f$ maps $f$ to its equivalence class in $W_2^s(S)$.} 
$$
\|f\|_{W^s_2(S)} := \Big( \sum_{ | \alpha | \leq s} \| \mathrm{D}^\alpha f \|^2 _{L^2(S)} \Big)^{\frac{1}{2}} .
$$
In order to define a Sobolev space on a Riemannian manifold $(M,g)$, let $(U_i, \phi_i)$ be a  $C^\infty$ atlas such that the $(U_i)$ form an open cover of $M$ and $(\rho_i)$ a partition of unity subordinate to $(U_i)$.
Then let $W_2^s(M)$ be the set of real-valued functions on $M$ for which the norm
\begin{eqnarray*}
\|f\|_{W^s_2(M)} & := & \Big( \sum_i \| (\rho_i f)\circ \phi^{-1}_i \|_ {W^s_2(\mathbb R^d)}^2 \Big)^{\frac{1}{2}}
\end{eqnarray*}
is finite.
It can be shown that $W_2^s(M)$ is a RKHS \citep{Fuselier2012}.
Note that the norm depends on the choice of atlas  and partition of unity. Different choices lead to different norms, however these are all equivalent \citep{Fuselier2012}. 
To avoid confusion, we fix a specific atlas and partition of unity in the sequel.
Further discussion can be found in \cite{DeVito2019}.
The restriction to compact manifolds is fundamental to our analysis, as Sobolev norms on general manifolds are not equivalent,  see section 3.2 
\citep{grosse2013sobolev}.
This completes our mathematical background.

\section{Results} \label{subsec: statement}

In this section our novel theoretical results are stated.
First in \Cref{subsec: stein pair}, we present and discuss the technical conditions that will be assumed.
In \Cref{subsec: comput MT} we explain how the choice of regularisers $R_1$, $R_2$ in \eqref{eq: Stein kernel method} allows the estimator $P_X(f)$ in \eqref{eq: Stein method} to be computed.
\Cref{subsec: connect to KSD} explains how the error of the Stein kernel interpolant can be quantified.
Then in \Cref{subsec: equiv kernels} we establish an equivalence between the Stein kernel interpolant and approximation using Sobolev spaces on a Riemannian manifold.
Finally, in \Cref{subsec: fsseb}, explicit sufficient conditions for the convergence of the estimator $P_X(f)$ are established.

\subsection{Stein Pair on a Riemannian Manifold} \label{subsec: stein pair}

The first task in this section is present and discuss the technical conditions that will be assumed on $\mathcal{H}$ and $\tau$ to ensure that $(\mathcal{H},\tau)$ is a Stein pair for $P$.
Two normed spaces $(X,\|\cdot\|_X)$, $(Y,\|\cdot\|_Y)$ are said to be \emph{equivalent} if the sets $X$ and $Y$ are equal and if there exists $0 < C < \infty$ such that $C^{-1} \|x\|_X \leq \|x\|_Y \leq C \|x\|_X$ for all $x \in X$.

\begin{assumption} \label{asm: Stein class}
Let $k : M \times M \rightarrow \mathbb{R}$ be positive definite, with $\mathcal{H}(k)$ equivalent to $W_2^{s+2}(M)$ for some $s > d/2$.
\end{assumption}

In this work a density $p$ for the distribution $P$ with respect to the Riemannian volume measure $V$ is required.
It will be assumed that $p$ is positive on $M$ and that a number of derivatives of $p$ exist:

\begin{assumption} \label{ass: density}
The function $\log p : M \rightarrow \mathbb{R}$ is $C^{s+1}(M,\mathbb{R})$, for the same exponent $s$ introduced in \Cref{asm: Stein class}.
\end{assumption}

The Stein pair that we consider is based on the Riemannian manifold generalisation of the original differential operator considered in \cite{Assaraf1999}:
\begin{assumption} \label{asm: Stein operator}
The operator $\tau : \mathcal{H} \rightarrow C^s(M,\mathbb{R})$ is the second order differential operator $\tau h = \nabla \cdot(p \nabla h) / p$.
\end{assumption}
\noindent Assumptions \ref{asm: Stein class} and \ref{ass: density} ensure that $\tau$ in \Cref{asm: Stein operator} is well-defined.
Indeed, from the product rule 
$$
\tau h = \frac{p \nabla \cdot \nabla h + g(\nabla p, \nabla h)}{p} \; = \; \Delta h + g \big( \nabla \log p,  \nabla h \big) ,
$$
from \Cref{ass: density} $\nabla \log p \in C^s(M,\mathbb{R})$,
and from tthe Sobolev embedding theorem $\Delta h \in C^s(M,\mathbb{R})$ whenever $h \in W_2^{s+2}(M)$ and $s > d/2$ \cite[Theorem 2.7]{Hebey2000}, which is \Cref{asm: Stein class}.
In what follows, the operator in \Cref{asm: Stein operator} will be called the \emph{Riemann--Stein operator} due to its suitability for Stein's method on a Riemannian manifold.
The Riemann--Stein operator is an example of a weighted Laplacian operator; see \cite{Grigor'yan2006}.
This operator coincides with the operator studied in the concurrent work of \cite{Liu2017,xu2020stein,Le2020}.

\begin{assumption} \label{asm: BC}
If $(M,g)$ is a Riemannian manifold with boundary $\partial M$, then 
\begin{eqnarray*}
\int_{\partial M} g \big( p \nabla h,\bm{n} \big) \; i_{\bm{n}} \mathrm{d}V  & = & 0 \qquad \forall h \in \mathcal{H}.
\end{eqnarray*}
\end{assumption}
\noindent For a manifold with boundary, the boundary condition in \Cref{asm: BC} is either automatically satisfied if $p$ vanishes on $\partial M$ or must be enforced through a suitable restriction on $\mathcal{H}$.
On the other hand, if $M$ is a closed manifold, then no assumption is required.

\begin{proposition} \label{prop: WSO}
$(\mathcal{H},\tau)$ is a Stein pair for $P$.
\end{proposition}
\begin{proof}
From \Cref{asm: Stein class,ass: density} the operator $\tau$ in \Cref{asm: Stein operator} is well-defined.
Thus for $h \in \mathcal{H}$, using the above assumptions and the divergence theorem on a Riemannian manifold \citep[see e.g.][]{Szekeres2004course}, we find
\begin{eqnarray}
 \int_M \tau h \; \mathrm{d}P & = &  \int_M \nabla \cdot(p \nabla h) \; \mathrm{d}V \nonumber \\
& = & \left\{ \begin{array}{ll} \int_{\partial M} g \big(  p \nabla h, \bm  n) \;  i_{\bm n}\mathrm{d}V & \qquad \text{if } M \text{ is a manifold with boundary} \\
0 & \qquad \text{if } M \text{ is a closed manifold.} \end{array} \right.   \label{eq: bdr intgrl}
\end{eqnarray}
Finally, in the case of a manifold with boundary, \Cref{asm: BC} ensures the boundary integral in \eqref{eq: bdr intgrl} is zero, as required.
\end{proof}

\subsection{Exact Computation of \eqref{eq: Stein method}} \label{subsec: comput MT}

Our assumptions also ensure that the estimator $P_X(f)$ can be computed.
The computations are analogous to those of traditional kriging \citep{Stein2012}, albeit based on a non-standard, non-radial kernel.
It will be assumed that the point set $X \subset M$ has already been generated.
No specific requirements are needed on $X$ in order for the estimator to be computed, however it will be assumed that its elements are distinct.
First, define the function $k_P : M \times M \rightarrow \mathbb{R}$ as $k_P(\bm{x},\bm{x}') := \tau' \tau k(\bm{x},\bm{x}')$ and, for $\sigma > 0$, define the function $k_{P,\sigma} : M \times M \rightarrow \mathbb{R}$ as $k_{P,\sigma}(\bm{x},\bm{x}') = \sigma^2 + k_P(\bm{x},\bm{x}')$, where in each case the Riemann--Stein operators $\tau$ and $\tau'$ act, respectively, on the first and second argument of the kernel.
The representer theorem, together with basic properties of RKHS, implies that a function $f_X = \xi + \tau h$ that solves \eqref{eq: Stein method} is an element of $\mathcal{H}(k_{P,\sigma})$ and identical calculations to those in \cite{Oates2017} establish that
\begin{eqnarray}
P_X(f) & = & \sigma^2 \mathbf{1}^\top \mathbf{K}_{P,\sigma}^{-1} \mathbf{f} \label{eq: the estimator}
\end{eqnarray}
where $[\mathbf{K}_{P,\sigma}]_{ij} := k_{P,\sigma}(\bm{x}_i,\bm{x}_j)$ and $[\mathbf{f}]_i := f(\bm{x}_i)$.
The requirement that elements of $X$ are distinct, together with the fact that $k_{P,\sigma}$ is equivalent to a known positive definite kernel (see \Cref{thm: equivalent kernels}), ensure that $\mathbf{K}_{P,\sigma}$ is non-singular.
The computation of \eqref{eq: the estimator} is associated with a $O(n^3)$ cost, which can be contrasted with the $O(n)$ cost of generating a point set $X$ using MCMC.
An increased cost (e.g. relative to MCMC) is typical for methods with accelerated convergence and can be justified when the convergence rate is sufficiently fast relative to the computational cost.

From \eqref{eq: the estimator}, the operator $P_X$ is seen to be linear.
In particular, it is recognised as a weighted \emph{cubature rule} $P_X(f) = \sum_{i=1}^n w_i f(\bm{x}_i)$ with weights $\bm{w} = [w_1,\dots,w_n]^\top$ the solution to $\mathbf{K}_{P,\sigma} \bm{w} = \sigma^2 \mathbf{1}$.
In practice, the parameter $\sigma$ can either be set in a data-driven manner or eliminated altogether in such a way that constant functions are exactly integrated:

\begin{proposition} \label{thm: limit}
Let $\mathbf{K}_P$ denote the $n \times n$ matrix with entries $k_P(\bm{x}_i,\bm{x}_j)$.
Then
\begin{eqnarray}
\lim_{\sigma \rightarrow \infty} P_X(f) & = & \left( \frac{\mathbf{K}_P^{-1} \mathbf{1}}{\mathbf{1}^\top \mathbf{K}_P^{-1} \mathbf{1}} \right)^\top \mathbf{f} . \label{eq: estimator}
\end{eqnarray}
\end{proposition}
\begin{proof} 
Note that $\mathbf{K}_{P,\sigma} = \sigma^2 \mathbf{1} \mathbf{1}^\top + \mathbf{K}_P$.
The proof is then an application of the Woodbury matrix inversion formula, which can be used to deduce that
\begin{eqnarray*}
P_X(f) \; = \; \sigma^2 \mathbf{1}^\top (\sigma^2 \mathbf{1} \mathbf{1}^\top + \mathbf{K}_P)^{-1} \mathbf{f} \; = \; \frac{\mathbf{1}^\top \mathbf{K}_P^{-1} \mathbf{f}}{\sigma^{-2} + \mathbf{1}^\top \mathbf{K}_P^{-1} \mathbf{1}} 
\end{eqnarray*}
from which the result is immediately established.
\end{proof}

\subsection{Connection to Kernel Stein Discrepancy} \label{subsec: connect to KSD}

In addition to returning an estimate for the integral, the Riemann--Stein kernel method is accompanied by a parsimonious error assessment.
Indeed, from the Moore--Aronszajn theorem, the kernel $k_{P,\sigma}$ induces a RKHS, denoted $\mathcal{H}(k_{P,\sigma})$, and the \emph{worst case error} of the cubature rule in \eqref{eq: the estimator} in the unit ball of $\mathcal{H}(k_{P,\sigma})$ is given by
\begin{eqnarray}
\sup \left\{ | P_X(f) - \textstyle \int_M f \mathrm{d}P | \; : \; \|f\|_{\mathcal{H}(k_{P,\sigma})} \leq 1 \right\} & = & (\sigma^{-2} + \mathbf{1}^\top \mathbf{K}_P^{-1} \mathbf{1})^{-1/2}  . \label{eq: wce}
\end{eqnarray}
The expression $(\mathbf{1}^\top \mathbf{K}_P^{-1} \mathbf{1})^{-1/2}$, which is recovered in the $\sigma \rightarrow \infty$ limit, can be interpreted as a {\it kernel Stein discrepancy} \citep[KSD;][]{Chwialkowski2016,Liu2016b,Gorham2017}
\begin{eqnarray}
\text{KSD}\left( \sum_{i=1}^n w_i \delta(\bm{x}_i) , P \right) \; := \; \sqrt{ \sum_{i,j = 1}^n w_i w_j k_P(\bm{x}_i , \bm{x}_j) } \label{eq: KSD general}
\end{eqnarray}
for the specific choice of weights
\begin{eqnarray}
\bm{w} & = & \frac{\mathbf{K}_P^{-1} \mathbf{1}}{\mathbf{1}^\top \mathbf{K}_P^{-1} \mathbf{1}}  \label{eq: ap weights}
\end{eqnarray}
that arise in \Cref{thm: limit}.
Indeed, minimisation of \eqref{eq: KSD general} over the weights $\bm{w}$ subject to the non-degeneracy constraint $\bm{1}^\top \bm{w} = 1$ leads to \eqref{eq: ap weights}, so that these weights are in a sense optimal.
Under certain conditions on $P$ and $k$, KSD controls the standard notion of weak convergence of the weighted empirical measure $\sum_{i=1}^n w_i \delta(\bm{x}_i)$ to the target $P$.
Sufficient conditions for the case $M = \mathbb{R}^d$ and a first order differential operator were established in \cite{Gorham2017,Chen2018}.

\subsection{First Main Result: Stein and Sobolev Kernels are Equivalent} \label{subsec: equiv kernels}

Our analysis of the Stein kernel interpolant is rooted in a more fundamental result that establishes that the reproducing kernel Hilbert space associated to $k_{P,\sigma}$ is equivalent to a standard Sobolev space on $M$:

\begin{theorem}[First Main Result] \label{thm: equivalent kernels}
The reproducing kernel Hilbert spaces $\mathcal{H}(k_{P,\sigma})$ and $\mathcal{H}(k)$ are equivalent, meaning that they are equal as sets and that there exists a constant $0 < C < \infty$ such that $C^{-1} \|x\|_{\mathcal{H}(k)} \leq \|x\|_{\mathcal{H}(k_{P,\sigma})} \leq C \|x\|_{\mathcal{H}(k)}$, for all $x \in \mathcal{H}_k$.
\end{theorem}

\noindent All proofs are provided in Appendix \ref{sec: proofs}.
Only one direction of this equivalence had been shown in earlier work; the embedding of $\mathcal{H}(k_{P,\sigma})$ into $\mathcal{H}(k)$ was shown in the Euclidean setting in \cite{Oates2018}.
The main effort in our proof is to establish the converse direction; i.e. that the \emph{Stein equation} $f = \xi + \tau h$ has a solution $h \in \mathcal{H}(k)$ and $\xi \in \mathcal H(\sigma^2)$ for each $f \in W_2^s(M)$.
The argument makes use of the theory of partial differential equations on a Riemannian manifold.

\Cref{thm: equivalent kernels} implies that the worst case error in \eqref{eq: wce} is equivalent to
\begin{eqnarray}
\sup \left\{ | P_X(f) - \textstyle\int_M f \mathrm{d}P | \; : \; \|f\|_{W_2^s(M)} \leq 1 \right\} , \label{eq: Sobolev wce}
\end{eqnarray}
the worst case error over the unit ball of $W_2^s(M)$.
In particular, \eqref{eq: wce} converges to $0$ only when $\sum_{i=1}^n w_i \delta(\bm{x}_i)$ converges weakly to $P$, since equivalent normed spaces have equivalent dual spaces and the \emph{Sobolev discrepancy} in \eqref{eq: Sobolev wce} controls weak convergence to $P$ (see the proof of \Cref{cor: KSD Sobolev}).
Likewise, since $\mathcal{H}(k_P)$ is equivalent to the quotient space $\mathcal{H}(k_{P,\sigma}) / \mathcal{H}(\sigma^2)$ and $\mathcal{H}(\sigma^2)$ is just a space of constant functions on $M$, it follows that $\mathcal{H}(k_P)$ is equivalent to the quotient space $W_2^s(M) / \mathcal{H}(\sigma^2)$.
This provides new insight into KSD; we have shown that the KSD in \eqref{eq: KSD general} is equivalent to the worst case integration error over the unit ball in $W_2^s(M) / \mathcal{H}(\sigma^2)$.

\begin{corollary} \label{cor: KSD Sobolev}
For any sequence $(P_n)_{n \in \mathbb{N}} \subset \mathcal{P}(M)$ of discrete measures $P_n := \sum_{i=1}^n w_i^{(n)} \delta(\bm{x}_i^{(n)})$ with $\sum_{i=1}^n w_i^{(n)} = 1$ and $\bm{x}_i^{(n)} \in M$, we have $\text{\normalfont KSD}( P_n , P )  \rightarrow 0 $ if and only if $P_n$ converges weakly to $P$.
\end{corollary}

\subsection{Second Main Result: Error Bounds for the Riemann--Stein Kernel Method} \label{subsec: fsseb}

In this section we demonstrate how \Cref{thm: equivalent kernels} leads to explicit finite-sample-size error bounds for the Riemann--Stein kernel method.
To state our result, the quality of a point set $X = \{\bm{x}_i\}_{i=1}^n$ must be quantified.
For the purposes of this work, we require that $X$ covers the manifold $M$ in the following sense \citep[see e.g.][]{Scheuerer2013}:
\begin{definition}
The \emph{fill distance} of a set $X = \{\bm{x}_i\}_{i=1}^n \subset M$  of points in $M$ is defined as $h_X := \max_{\bm{x} \in M} \; \min_{i = 1,\dots,n} d_M(\bm{x},\bm{x}_i)$, where $d_M$ is the geodesic distance on the Riemannian manifold.
\end{definition}

The distribution $P$ defines a linear operator on $C(M,\mathbb{R})$ that we denote, in a small abuse of notation, by $P(f) := \int_M f \mathrm{d}P$.
A general error bound can now be stated:

\begin{theorem} \label{thm: main result}
For all $f \in W_2^s(M)$,
\begin{eqnarray}
|P_X(f) - P(f)| & \leq & C_s h_X^s \|f\|_{W_2^s(M)} ,  \label{eq: specification of result}
\end{eqnarray}
where $C_s$ is a constant depending on $s$ but independent of $f$ and of the point set $X$.
\end{theorem}

\noindent Thus when the points in $X$ cover the manifold $M$, in the sense that the fill distance $h_X$ is small, the estimator $P_X(f)$ is an accurate approximation to the true integral.
For $P$ equivalent (in the sense of measures) to the natural volume measure $V$, an information-theoretic lower bound on the worst case error for any estimator $P_X$, based on a size $n$ point set $X$, is $C n^{-s/d}$, for some $C > 0$ \citep{Brandolini2014}.
This shows that the Riemann--Stein estimator $P_X$ is rate-optimal whenever the point set $X$ is selected such that the fill-distance is asymptotically minimised, i.e. $h_X = O(n^{- 1/d})$.
This is in principle a weak requirement, as under suitable conditions even independently sampled $\bm{x}_i \sim \bar{V}$, where $\bar{V}$ denotes the normalised Riemannian measure on the manifold, achieves this rate up to a logarithmic factor; see \cite[][Thm. 3.2, Cor. 3.3]{Reznikov2015} and \cite{Ehler2017}.

The most popular approaches to Bayesian computation are based on sampling, in particular MCMC.
Next we present the consequences of \Cref{thm: main result} in the case where the point set $X$ arises as MCMC output.
For measurable $A \subseteq M$ we let $P(A) := P(1_A)$ where $1_A(\bm{x}) = 1$ if $\bm{x} \in M$, otherwise $1_A(\bm{x}) = 0$.
Recall that notions of geometric and uniform ergodicity coincide on a compact state space when the invariant measure $P$ is equivalent (in the sense of measures) to $V$; we therefore describe a Markov chain $(\bm{x}_i)_{i \in \mathbb{N}}$ with $n$th step transition kernel $P^n$ and invariant distribution $P$ simply as \emph{ergodic} if there exists a finite constant $C$ and a number $0 \leq \rho < 1$ such that $|P^n(\bm{x}_0,A) - P(A)| \leq C \rho^n$ for all $\bm{x}_0 \in M$ and all measurable $A \subseteq M$.
The reader is referred to \cite{Meyn2012} for background.

\begin{corollary}[Second Main Result] \label{cor: MCMC}
Suppose that $(\bm{x}_i)_{i \in \mathbb{N}}$ is an ergodic Markov chain with invariant distribution $P$, initialised at an arbitrary point $\bm{x}_0 \in M$.
Let $X = \{\bm{x}_i\}_{i=1}^n$ and denote expectation with respect to the sampling distribution of $X$ as $\mathbb{E}$.
Then
\begin{eqnarray*}
\mathbb{E} |P_X(f) - P(f)| & \leq & C_s' n^{-\frac{s}{d}} \log(n)^{\frac{s}{d}} \|f\|_{W_2^s(M)} ,
\end{eqnarray*} 
where $C_s'$ is a constant depending on $s$ but independent of $f$.
\end{corollary}
\noindent 
This result shows that the Riemann--Stein kernel method provides an asymptotically more accurate approximation of $P(f)$ compared to using MCMC whenever $s > d/2$ (which was also assumed in \Cref{asm: Stein class}).
Note that the constant $C_s'$ is dependent on $P$ through the infimum of $p(\bm{x})$ on $\bm{x} \in M$ and could be large when $P$ is highly concentrated.
If the $O(n^3)$ computational cost is also taken into account, the Riemann--Stein method asymptotically out-performs MCMC on an error-per-cost basis only when $s > d$, following an identical argument to Sec. 1 of \cite{Oates2018}.
As usual, the case of high-dimensional manifolds (i.e. $d$ large) is likely to challenge any interpolation-based method unless stronger assumptions can be made on the integrand.

\section{Discussion} \label{sec: discussion}

This paper adds to the growing literature on computational uses for Stein's method.
Our contribution provides both a manifold generalisation and a refined theoretical analysis of the Stein kernel interpolant.
This was achieved by characterising kernel Stein discrepancy as a Sobolev discrepancy in the context of a compact manifold, and this characterisation may be of independent interest.
For example, \cite{mroueh2017sobolev} proposed to use Sobolev discrepancy to train generative adversarial networks (GAN); our result justifies the alternative use of KSD for training GANs in a manifold context.
See also \cite{mroueh2018regularized,mroueh2019sobolev,arbel2019maximum,marzo2019discrepancy} for uses of Sobolev discrepancy, which could instead be achieved using KSD.
Our main contribution was theoretical, but for completeness our results are empirically verified on $M = \mathbb{S}^2$ in the electronic supplement.

Some extensions of this work could include; additional constructions to circumvent the need for gradient information on the target \citep{Han2018}; the simplification of computation when the density $p$ can be factorised \citep{Zhuo2018}; and extension to high- or infinite-dimensional spaces such as the Hilbert sphere $\mathbb{S}^\infty$ for functional data analysis on a manifold.

\section*{Acknowledgements}

CJO and MG were supported by the Lloyd's Register Foundation programme on data-centric engineering at the Alan Turing Institute, UK.
AB was supported by a Roth scholarship from the Department of Mathematics at Imperial College London, UK.
EP was partially supported by FONDECYT Grant [1170290], Chile, and by Iniciativa Cienti\'{i}fica Milenio - Minecon Nucleo Milenio MESCD.
MG was supported by the EPSRC grants [EP/K034154/1, EP/R018413/1, EP/P020720/1, EP/L014165/1], an EPSRC Established Career Fellowship [EP/J016934/1] and a Royal Academy of Engineering Research Chair in Data Centric Engineering. 
The authors are grateful for discussions with Andrew Duncan, Toni Karvonen, Chang Liu, Gustav Holzegel, Julio Delgado and Andrew Stuart.

\appendix

\section{Proofs of Theoretical Results} \label{sec: proofs}

This section contains proofs of the theoretical results presented in the main text.
To this end, there are two main theoretical challenges to be addressed:
First, it is necessary to establish that $k_{P,\sigma}$ is a valid kernel so that, from the Moore-Aronszajn theorem, $k_{P,\sigma}$ defines a RKHS.
This is addressed first in \Cref{subsec: stein rkhs}.
The RKHS $\mathcal{H}(k_{P,\sigma})$ will be called the \emph{Stein RKHS}.
The remainder of \Cref{subsec: stein rkhs} is devoted to proving \Cref{thm: equivalent kernels}, which establishes the equivalence of the Stein RKHS with a standard Sobolev space on the Riemannian manifold.
From this point onward, our interpolation error bounds are standard in the Sobolev space context and the proof for \Cref{thm: main result} is contained in \Cref{subsec: main proof sec}.
The corollary for MCMC are established in \Cref{subsec: cors for mcmc}.

\subsection{Proof of \Cref{thm: equivalent kernels}} \label{subsec: stein rkhs}

It is required to establish that $\mathcal{H}(k_{P,\sigma})$ is norm-equivalent to the Sobolev space $W_2^s(M)$.
This is performed in two parts, with a Sobolev embedding of the Stein RKHS  and a Stein embedding of the Sobolev RKHS.
For both parts we leverage results from the analysis of partial differential equations on a Riemannian manifold; our main reference is \cite{Grosse2017}.
The first result establishes how the kernel $k_P$ can be computed:

\begin{lemma}  \label{lem: compute kernel}
Let $\langle \cdot , \cdot \rangle_{\mathcal{H}(k)}$ denote the inner product in $\mathcal{H}(k)$.
Then $k_P(\bm{x},\bm{x}') \; := \; \tau' \tau k (\bm{x},\bm{x}') \; = \; \langle \tau k(\bm{x},\cdot) , \tau' k(\bm{x}',\cdot) \rangle_{\mathcal{H}(k)}$. 
In particular, $k_P$ is symmetric and semi-positive definite; i.e. $k_P$ is a kernel.
\end{lemma}
\begin{proof}
Since $k$ reproduces $W_2^\alpha(M)$ for $\alpha = s + 2 > (d/2) + 2$, it follows from the Sobolev embedding theorem that $k(\bm{x},\cdot) \in C^2(M,\mathbb{R})$.
Thus second order differential operators, such as $\tau$, can be applied to this function \citep[Lem. 4.34 of][]{Steinwart2008}.

Let $g^{-1}$ be the metric tensor on differential forms associated to $g$ by the musical isomorphism \citep{gallot1990riemannian}. 
Then $g(\nabla f, \nabla h) = g^{-1}(\mathrm{d}f , \mathrm{d}h)$ where $\mathrm{d}h$ is the linear functional on $T_{\bm{x}}M$ such that $\mathrm{d}h(v) = v(h)$. 
Recall that $\tau'$ is used to denote the action of a differential operator $\tau$ on $\bm{x}'$.
It follows that,
\begin{eqnarray*}
g\big(\nabla' \log p , \nabla' g(\nabla \log p, \nabla k) \big) & = & g^{-1}\big( \mathrm{d}' \log p, \mathrm{d}' g^{-1}( \mathrm{d} \log p, \mathrm{d} k) \big) \\
& = & \sum_{i, j, r, l} g^{ij}(y) g^{rl}(x) \partial_{y^i} \log p(y) \partial_{x^r} \log p(x) \partial_{y^j} \partial_{x^l} k(x,y) \\
& = & \nabla \log p \nabla' \log p (k).
\end{eqnarray*}
Here $\nabla \log p(k)$ is the function on $M$ which maps $x$ to $\nabla_x \log p(k) \in \mathbb R$, where $\nabla_x \log p \in T_xM$ is the gradient vector. Thus 
$\nabla \log p \nabla' \log p (k) = \nabla \log p \big( \nabla' \log p (k) \big)$ is defined.
Thus 
\begin{eqnarray*}
\tau' \tau k(\bm{x} , \bm{x}') & = & \nabla \log p \nabla' \log p (k)+g^{-1} ( \mathrm{d} \log p, \mathrm{d} \Delta' k)+g^{-1} ( \mathrm{d}' \log p, \mathrm{d}' \Delta k) + \Delta \Delta' k \\
& = & \langle \tau k(\bm{x},\cdot), \tau' k(\bm{x}',\cdot) \rangle_{\mathcal{H}(k)}
\end{eqnarray*}
where local coordinates verify that the last equality is established.

Finally, observe that the map $\phi : M \rightarrow \mathcal{H}(k)$, defined as $\phi(\bm{x}) := \tau k(\bm{x},\cdot)$, is a \emph{feature map} for $k_P$, meaning that $k_P(\bm{x},\bm{x}') = \langle \phi(\bm{x}) , \phi(\bm{x}') \rangle_{\mathcal{H}(k)}$ for all $\bm{x},\bm{x}' \in M$.
It follows that $k_P$ is symmetric and semi-positive definite; indeed if $\{w_i\}_{i=1}^n \subset \mathbb{R}$ and $\{\bm{x}_i\}_{i=1}^n \subset M$ then $\sum_{i=1}^n \sum_{j=1}^n w_i w_j k_P(\bm{x}_i , \bm{x}_j) = \langle \sum_{i=1}^n w_i \phi(\bm{x}_i) , \sum_{j=1}^n w_j \phi(\bm{x}_j) \rangle_{\mathcal{H}(k)} = \| \sum_{i=1}^n w_i \phi(\bm{x}_i) \|_{\mathcal{H}(k)}^2 \geq 0$ as required.
\end{proof}

\begin{lemma}
$k_{P,\sigma}$ is symmetric and semi-positive definite; i.e. $k_{P,\sigma}$ is a kernel.
\end{lemma}
\begin{proof}
Let $\Phi := \mathbb{R} \times \mathcal{H}(k)$ denote the Hilbert space with inner product $\langle (c_1,h_1) , (c_2,h_2) \rangle_\Phi := c_1 c_2 + \langle h_1 , h_2 \rangle_{\mathcal{H}(k)}$ for all $c_1,c_2 \in \mathbb{R}$ and all $h_1,h_2 \in \mathcal{H}(k)$.
Denote the associated norm $\| \cdot \|_\Phi$.
Then the map $\phi : M \rightarrow \Phi$, defined as $\phi(\bm{x}) := [\sigma, \tau k(\bm{x},\cdot)]$, is a \emph{feature map} for $k_{P,\sigma}$, meaning that $k_{P,\sigma}(\bm{x},\bm{x}') = \langle \phi(\bm{x}) , \phi(\bm{x}') \rangle_\Phi$ for all $\bm{x},\bm{x}' \in M$.
It follows that $k_{P,\sigma}$ is symmetric and semi-positive definite, as explained in the proof of \Cref{lem: compute kernel}.
\end{proof}

The Stein RKHS $\mathcal{H}(k_{P,\sigma})$ has the form $\mathcal{H}(\sigma^2) \oplus \mathcal{H}(k_P)$ where $\mathcal{H}(\sigma^2)$ is the RKHS with constant kernel $\sigma^2$.
Moreover, from Thm. 4.21 of \cite{Steinwart2008}, for $\zeta \in \mathcal{H}(k_P)$
\begin{eqnarray*}
\|\zeta \|_{\mathcal{H}(k_P)} & = & \inf \{ \|h\|_{\mathcal{H}(k)} : h \in \mathcal{H}, \; \tau h = \zeta \} .
\end{eqnarray*}
Thus, from Thm. 5 of \cite{Berlinet2011},
\begin{eqnarray}
\|f\|_{\mathcal{H}(k_{P,\sigma})}^2 & = & \inf \{ \sigma^{-2} \xi^2 + \|\zeta \|_{\mathcal{H}(k_P)}^2 : f = \xi + \zeta, \; \xi \in \mathbb{R}, \; \zeta \in \mathcal{H}(k_P) \} \label{eq: RKHS sum 2} \\
& = & \inf \{ \sigma^{-2} \xi^2 + \|h\|_{\mathcal{H}(k)}^2 : f = \xi + \tau h, \; \xi \in \mathbb{R}, \; h \in \mathcal{H}(k) \} . \label{eq: RKHS sum}
\end{eqnarray}
Note that, since the only constant function in $\mathcal{H}(k_P)$ is the zero function, the set over which the infimum is sought in \eqref{eq: RKHS sum 2} in fact contains a single element.
Let $\hat{f} = \hat{\xi} + \tau \hat{h}$ where
\begin{eqnarray}
(\hat{\xi},\hat{h}) & = & \arginf_{\xi \in \mathbb{R} , \; h \in \mathcal{H}(k)} \; \sum_{i=1}^n \Big( \xi + \tau h(\bm{x}_i) - f(\bm{x}_i) \Big)^2 + R_1(\xi) + R_2(h)  
\end{eqnarray}
and $R_1$, $R_2$ are defined in \eqref{eq: Stein kernel method}.
The $(\hat{\xi} , \hat{h})$ exist and are unique from the representer theorem \citep{Scholkopf2001}.
To understand the Riemann--Stein kernel method it suffices to study the convergence of $\hat{f}$ to $f$ in the Stein RKHS.
To this end, note at this point that the proof of the representer theorem shows that $\hat{f}$ is the orthogonal projection of $f$ onto the span of $\{k_{P,\sigma}(\cdot,\bm{x}_1), \ldots, k_{P,\sigma}(\cdot,\bm{x}_n) \}$ in $\mathcal{H}(k_{P,\sigma})$, and thus we have
\begin{eqnarray}
\| f - \hat{f} \|_{\mathcal{H}(k_{P,\sigma})} & \leq & \| f \|_{\mathcal{H}(k_{P,\sigma})}, \label{eq: best approximation property}
\end{eqnarray}
the so-called \emph{best approximation property}.

\Cref{thm: equivalent kernels} will now be proven in two parts; first we establish a Sobolev embedding of the Stein RKHS in \Cref{lem: Sobolev embed Stein} using a derivative counting argument, then a Stein embedding of the Sobolev RKHS is established in \Cref{lem: Stein embed Sobolev} using results from the theory of partical differential equations on a Riemannian manifold.

\begin{lemma}[Stein Embedding of Sobolev RKHS] \label{lem: Sobolev embed Stein}
The space $\mathcal{H}(k_{P,\sigma})$ is continuously embedded in $W_2^s(M)$.
i.e.  for some constant $0 < C < \infty$, we have $\|f\|_{W_2^s(M)} \leq C \|f\|_{\mathcal{H}(k_{P,\sigma})}$.
\end{lemma}
\begin{proof}

The elements of $\mathcal{H}(k_{P,\sigma})$ are all of the form $f = \xi + \tau h$ for $\xi \in \mathbb{R}$, $h \in \mathcal{H}(k)$, and therefore it is sufficient to show that there is a finite constant $C$ such that $\|\xi + \tau h \|_{W_2^s(M)} \leq C \|f\|_{\mathcal{H}(k_{P,\sigma})}$.
From the triangle inequality and the definition of Sobolev norm, $\|\xi + \tau h \|_{W_2^s(M)} \leq C_1( \xi + \|\tau h \|_{W_2^s(M)})$, where the constant $C_1$ depends on the volume of $M$.
Now, we also know that $\|f\|_{\mathcal{H}(k_{P,\sigma})}^2 = \sigma^{-2} \xi^2 + \|h'\|_{\mathcal{H}(k)}^2$ where $h'$ is the element of $\mathcal{H}(k)$ that minimises $\|h'\|_{\mathcal{H}(k)}$ subject to $\tau(h'-h) = 0$.
The result therefore follows if $\|\tau h'\|_{W_2^s(M)} \leq C \|h'\|_{W_2^{s+2}(M)}$. 
This is a consequence of Lemma 2.4 in \cite{Grosse2017}, and holds for the differential operator $\tau = (1/p) D_a$ whenever the section $a=pg(\cdot, \cdot)$ belongs to $W_\infty^s(M,T^*M\otimes T^*M)$; see \cite{baez1994gauge}. 
Here $D_a := \nabla \cdot (p \nabla )$ is the differential operator associated to $a$, and $a$ is defined as the map that sends a point $\bm{x} \in M$ to the bilinear form $p(\bm{x})g_{\bm{x}}:T_{\bm{x}} M\times T_{\bm{x}} M \to \mathbb R$.
Since $M$ is compact and $g$ is a smooth section, the condition holds for example whenever $p\in C^{s+1}(M,\mathbb{R})$.
The final condition is implied by \Cref{ass: density}. 
\end{proof}

The converse of \Cref{lem: Sobolev embed Stein} is now established:

\begin{lemma}[Sobolev Embedding of Stein RKHS] \label{lem: Stein embed Sobolev}
The space $W_2^s(M)$ is continuously embedded in $\mathcal{H}(k_{P,\sigma})$.
i.e. for some constant $0 < C < \infty$, we have $\|f\|_{\mathcal{H}(k_{P,\sigma})} \leq C \|f\|_{W_2^s(M)}$.
\end{lemma}
\begin{proof}
Let $f \in W_2^s(M)$ and consider solutions $\xi \in \mathbb{R}$, $h \in W_2^{s+2}(M)$ to the partial differential equation
\begin{eqnarray}
\tau h + \xi & = & f \qquad \text{in } M \label{eq: PDE} \\
g(\nabla h , \bm{n} ) & = & 0 \qquad \text{on } \partial M,  \nonumber
\end{eqnarray}
sometimes called the \emph{Stein equation} \citep{Ley2017}. 
If the manifold has no boundary we simply drop the boundary condition and Stein equation reduces to
$\tau h + \xi  = f$.
Note that for the equation to be well-defined
we must have $\xi = P(f)$
 due to \Cref{asm: BC}. 
We re-write the first equation in \eqref{eq: PDE} as  $D_a(h) = p(f-\xi)$ with the notation from the proof of Lemma \ref{lem: Sobolev embed Stein}.
 In order to obtain existence of solution we rely on the Fredholm alternative, according to which a solution exists if and only if $p(f-\xi)$ is orthogonal in $L^2(M,\mathrm{d}V)$ to the kernel of the adjoint in $L^2(M,\mathrm{d}V)$ of $D_a$, see for example Corollary 10.4.10
 \cite{nicolaescu2007lectures}.
 Observe that the adjoint of $D_a$ is $D_a$ itself, since by the divergence theorem, for any $h,u \in C^{\infty}(M)$
 \begin{align*}
 \int D_a(h) u \mathrm{d}V 
   &= \int u \nabla \cdot (p\nabla h) \mathrm{d}V = - \int p \nabla h(u) \mathrm{d}V=- \int p \nabla h\cdot \nabla u \mathrm{d}V\\
&=
- \int p  \nabla u(h) \mathrm{d}V = \int h \nabla \cdot (p \nabla h) \mathrm{d}V= \int h  D_a (u) \mathrm{d}V.
 \end{align*}
 Moreover the kernel of $D_a$ consists only of constant functions, since if $h$ belongs to $\text{ker} D_a$, then multiplying by $h$ and integrating yields
 $$D_a(h)=0 \quad \implies \quad  0=\int hD_a(h) \mathrm{d}V
 = \int h \nabla \cdot (p\nabla h) \mathrm{d}V=
 - \int  g(\nabla h,\nabla h) \mathrm{d}P,$$
 and the positive-definiteness of the Riemannian metric, $g(\nabla h,\nabla h) \geq 0$, implies that $g(\nabla h,\nabla h)$ is the zero function, which is equivalent to $\nabla h=0$, which is equivalent to $h=$ constant.
 Existence hence follows by noting that 
 $p(f-\xi)$ is orthogonal to the constant functions in $L^2(M,\mathrm{d}V)$, which is a consequence of the fact that its integral with respect to $P$ vanishes,
 $$ \int p(f-\xi)h \mathrm{d}V=h\int(f-\xi)\mathrm{d}P=0 \quad \forall \text{ contant functions } h.$$   

The existence of a solution implies that 
$W_2^s(M)$ is a subset of $\mathcal{H}(k_{P,\sigma})$, and it remains to prove that it is continuously embedded.
To do so we will apply Lemma 10.4.9 \cite{nicolaescu2007lectures}, which states that there exists $0 < C < \infty$ such that, for all $u \in W_2^{s+2}(M)$ that are orthogonal to $\text{ker} D_a$ in $L^2(M)$,   
$$ \|u\|_{W_2^{s+2}(M)}  \leq   C \| D_a(u)\|_{W_2^s(M)}.$$
Let $\overline{h}$ be a solution to the Stein equation, which exists from the above argument, and define
$h = \overline{h} - \int \overline{ h} \mathrm{d}V$.
 Then $h$ is a solution with vanishing integral and is thus orthogonal to the $\text{ker} D_a$.
Hence, for some finite constants $C_i$,
\begin{eqnarray*} 
\|h\|_{W_2^{s+2}(M)} & \leq  & C_1 \| p(f-\xi)\|_{W_2^s(M)}  \\ 
 & \leq  & C_2 \| f-\xi \|_{W_2^s(M)}  \\
 & \leq & C_2 \big (\| f \|_{W_2^s(M)}  +\|\xi \|_{W_2^s(M)} \big )
\end{eqnarray*}
where in a small abuse of notation $\xi$ denotes also the constant function with value $\xi$ and we have used the fact that the density $p$ is bounded above on $M$.
Using Jensen's inequality $ \|\xi \|_{W_2^s(M)} \leq C_3 | \xi | \leq C_3 \| f \|_{L^1(M)} \leq  C_4 \| f \|_{L^2(M)} \leq  C_5 \| f \|_{W^s_2(M)}$. 
Thus $ \|h\|_{W_2^{s+2}(M)} \leq C_6 \|f\|_{W_2^s(M)}$. 

To complete the proof, we have from \eqref{eq: RKHS sum} that 
\begin{eqnarray}
\|f\|_{\mathcal{H}(k_{P,\sigma})}^2 & \leq & \sigma^{-2} P(f)^2 + \|h\|_{\mathcal{H}(k)}^2 \label{eq: specific bound}
\end{eqnarray}
where $h$ is the unique element in $\mathcal{H}(k)$ that satisfies \eqref{eq: PDE}.
Now, from Jensen's inequality $P(f)^2 \leq P(f^2)$ and, from the definition of the Sobolev norm and the fact that $p$ is bounded above on $M$, we have $\|f\|_{L_2(p)} \leq C_7 \|f\|_{W_2^s(M)}$.
Moreover, by hypothesis $\mathcal{H}(k)$ is norm-equivalent to $W_2^{s+2}(M)$; i.e. $\|h\|_{\mathcal{H}(k)} \leq C_8 \|h\|_{W_2^{s+2}(M)}$ for some finite constant $C_8$.
Thus
\begin{eqnarray}
\|f\|_{\mathcal{H}(k_{P,\sigma})}^2 & \leq & (\sigma^{-2}C_7^2 + C_6^2 C_8^2) \|f\|_{W_2^s(M)}^2
\end{eqnarray}
as required, with $C = (\sigma^{-2}C_7^2 + C_6^2 C_8^2)^{\frac{1}{2}}$.
\end{proof}

Together, \Cref{lem: Sobolev embed Stein,lem: Stein embed Sobolev} establish that the Stein RKHS is equivalent to the Sobolev RKHS and thus \Cref{thm: equivalent kernels} is established.
This result, which is likely to be of independent interest, enables us to reduce the analysis of methods based on a Stein kernel to the analysis of methods based on a Sobolev kernel.

\subsection{Proof of \Cref{cor: KSD Sobolev}}

The argument from the main text establishes that $\text{KSD}(P_n,P) \rightarrow 0$ if and only if $\text{MMD}_k(P_n,P) \rightarrow 0$, where $\text{MMD}_k(Q,P) := \sup\{Q(f) - P(f) : \|f\|_{\mathcal{H}(k)} \leq 1\}$ is called the \textit{maximum mean discrepancy} associated with the kenel $k$ \citep{Smola2007}.
Recall that a kernel $k$ is said to be \textit{characteristic} if $P \mapsto \int_M k(\cdot,x) \mathrm{d}P(x)$ is an injective map into $\mathcal{H}(k)$.
Thm. 3.1 of \cite{SimonGabriel2020} ensures that, for a continuous, bounded and characteristic kernel $k$ on a compact Hausdorff space $M$, $\text{MMD}_k$ metrises weak convergence in $\mathcal{P}(M)$.
Thus we aim to establish the preconditions of this result.
In our setting, $M$ is a compact Hausdorff space and, from \Cref{asm: Stein class} and the Sobolev embedding theorem, $\mathcal{H}(k) \subset C(M,\mathbb{R})$.
The inclusion $\mathcal{H}(k) \subset C(M,\mathbb{R})$ implies $k$ is continuous and bounded, due to Cor. 3 in \cite{Simon2018}, so it remains only to argue that $k$ is a characteristic kernel.

Recall that a positive definite kernel $k : M \times M \rightarrow \mathbb{R}$ on a compact Hausdorff space $M$ is said to be \textit{$c$-universal} if $\mathcal{H}(k)$ is dense in $C(M,\mathbb{R})$ with respect to the uniform norm on $M$ \citep{Sriperumbudur2011}.
Moreover, recall that a $c$-universal kernel on a compact Hausdorff space is a characteristic kernel \cite[Thm. 6 in][]{Simon2018}.
Thus we must argue that $\mathcal{H}(k)$ in \Cref{asm: Stein class} is dense in $C(M,\mathbb{R})$ with respect to the uniform norm on $M$.
i.e. that $W_2^{s+2}(M)$ is dense in $C(M,\mathbb{R})$.
This is a standard fact, which follows from the fact that $W_2^{s+2}(M)$ is a sub-algebra of $C(M,\mathbb{R})$ that separates points and vanishes nowhere, together with the Stone--Weirstrass theorem on a compact domain:

\begin{theorem}[Stone--Weirstrass on compact $M$]
Let $M$ be a compact Hausdorff space and $F$ be a sub-algebra of $C(M,\mathbb{R})$.
Then $F$ is dense in $C(M,\mathbb{R})$ (equipped with the uniform norm on $M$) if and only if $F$ separates points ($\forall x,y \in M$, $\exists f \in F$ s.t. $f(x) \neq f(y)$) and vanishes nowhere ($\forall x \in M$, $\exists f \in F$ s.t. $f(x) \neq 0$).
\end{theorem}
\begin{proof}
A special case of the more general result for locally compact spaces due to \cite{Stone1948}.
\end{proof}

\subsection{Proof of \Cref{thm: main result}} \label{subsec: main proof sec}

In light of \Cref{thm: equivalent kernels}, the proof of \Cref{thm: main result} is relatively standard.
Indeed, we have established that $\mathcal{H}(k_{P,\sigma})$ and $W_2^s(M)$ are equivalent, so the convergence of $\hat{f}$ to $f$ can be studied in the standard Sobolev space context.
The following \Cref{lem: Atlas} follows immediately from Prop. 7 and Thm. 8 in \cite{Fuselier2012}, together with the fact all Riemannian metrics (and their induced norms) are equivalent on a compact manifold, see page 22 of \cite{Hebey2000}. 

\begin{lemma} \label{lem: Atlas}
Let $M$ be a smooth, compact, $d$-dimensional Riemannian manifold.
Then there exists an atlas $(U_i,\phi_i)$ on $M$ and constants $C_1,C_2>0$ such that, if $\bm x, \bm y \in U_j$ for some $j$, then 
 \begin{eqnarray*}
  C_1 \| \phi_j(\bm x)-\phi_j(\bm y) \|_2  \; \leq \; d_M(\bm x,\bm y) \; \leq \; C_2 \| \phi_j(\bm x)-\phi_j(\bm y) \|_2
 \end{eqnarray*}
Moreover, let $h_{X\bigcap U_j} \big |_{U_j} := \sup_{\bm{x} \in U_j } \min_{\bm y \in X\bigcap U_j} d_M(\bm{x},\bm{y})$ denote the fill distance restricted to $U_j$.
Then, if we choose an atlas as above, there exists finite constants $h_0,C$ such that, if $X\subset M$ is a finite point set with $h_X < h_0$, then for all $U_j$
  \begin{eqnarray*}
 h_{X\bigcap U_j} \big |_{U_j} \; \leq \; C h_X .
 \end{eqnarray*}
\end{lemma}

The output of the Stein kernel method satisfies $P_X(f) = P( \hat{f})$ and in particular this means that
\begin{eqnarray*}
|P(f) - P_X(f)| = |P(f - \hat{f})| = \left| \int_M \big( f - \hat{f} \big) \mathrm{d}P \right| \leq \|f - \hat{f}\|_{L_2(P)} .
\end{eqnarray*}
Using \Cref{lem: Atlas}, and following the argument used in  Thm. 10  of \cite{Fuselier2012}, we have that 
\begin{eqnarray*}
\| f - \hat{f} \|_{ L_2(P) } & \leq & C h_X^s \| f - \hat{f}\|_{W_2^s(M)}
\end{eqnarray*}
Hence, for some finite constants $C_i$,
\begin{eqnarray*}
\| f - \hat{f} \|_{ W_2^s(M) } & \leq & C_1 \| f - \hat{f} \|_{ \mathcal{H}(k_{P,\sigma}) } \quad \text{(\Cref{lem: Sobolev embed Stein})} \\
& \leq & C_1 \|f\|_{\mathcal{H}(k_{P,\sigma})} \quad \text{(best approximation property; \eqref{eq: best approximation property})} \\
& \leq & C_2 \|f\|_{W_2^s(M)} \quad \text{(\Cref{lem: Stein embed Sobolev})} .
\end{eqnarray*}
Thus $|P(f) - P_X(f)| \leq C_3 h_X^s \|f\|_{W_2^s(M)}$, as required.
This completes the proof of \Cref{thm: main result}.

\subsection{Proof of \Cref{cor: MCMC}} \label{subsec: cors for mcmc}

The proof proceeds in two stages, and first we state and prove a result for independent samples that will serve as a stepping stone toward the main result:

\begin{corollary} \label{cor: MC}
Suppose that $(\bm{x}_i)_{i \in \mathbb{N}}$ is a sequence of independent samples from the distribution $P$.
Let $X = \{\bm{x}_i\}_{i=1}^n$ and denote expectation with respect to the sampling distribution of $X$ as $\mathbb{E}$.
Then
\begin{eqnarray*}
\mathbb{E} |P_X(f) - P(f)| & \leq & C_s' n^{-\frac{s}{d}} \log(n)^{\frac{s}{d}} \|f\|_{W_2^s(M)} ,
\end{eqnarray*} 
where $C_s'$ is a constant depending on $s$ but independent of $f$.
\end{corollary}
\begin{proof}[Proof of \Cref{cor: MC}]
First, let $(\bm{Y}_i : \Omega \to M)_{i \in \mathbb{N}}$ be a sequence of i.i.d. random variables distributed according to the normalised Riemannian measure $\bar{V}$, i.e., $\mathbb P \circ \bm{Y}_i^{-1} =\bar{V}$, where $\mathbb P$ is the probability measure on $\Omega$.
Let $Y^n = \{\bm{Y}_i\}_{i=1}^n$.
From \cite[][Thm. 3.2, Cor. 3.3]{Reznikov2015} we have that
\begin{eqnarray}
\mathbb{E} [h_{Y^n}^s] & \leq & C_s'' n^{-\frac{s}{d}} \log(n)^{\frac{s}{d}}  \label{eqn: Reznikov}
\end{eqnarray}
for some $C_s''$ a finite $s$-dependant constant.
Since $M$ is compact and $p$ is continuous (\Cref{ass: density}), the Radon-Nikodym derivative $\mathrm{d}P / \mathrm{d}\bar{V}$ can be bounded away from zero almost everywhere on $M$, say by $w > 0$, and without loss of generality we can assume that $w < 1$.
It follows that $P$  can be represented as a bivariate mixture, one of whose components is $\bar{V}$.
Specifically, $P=w\bar V+(1-w)Q$, where $Q$ has density 
\begin{eqnarray*}
 \frac{\mathrm{d}Q}{\mathrm{d}\bar{V}} \; = \; \frac{1}{1-w}\left( \frac{\mathrm{d}P}{\mathrm{d}\bar{V}} - w \right) .
\end{eqnarray*}

Let us introduce a Bernoulli random variable $ B \sim \text{Bernoulli}(w)$, and a random variable $\bm{Z}$ with conditional distributions
\begin{eqnarray*}
\bm{Z} | B & \sim & \left\{ \begin{array}{ll} \bar{V} & \text{if } B= 1 \\ Q & \text{if } B = 0 \end{array} \right. 
\end{eqnarray*}
It follows that the law of $\bm{Z}$ is $P$, since 
\begin{eqnarray*} 
\mathbb P \circ \bm{Z}^{-1}(A) & = &  \mathbb P( \bm{Z} \in A |B =1) \mathbb P(B=1)+\mathbb P(\bm{Z} \in A | B=0) \mathbb P(B=0) \\
& = &  \bar{V}(A) w + Q(A) (1-w) \\
& = & P(A), 
\end{eqnarray*}
 where $w = \mathbb P (B=1) $. Consider then such a sequence $(\bm{Z}_i)_{i \in \mathbb{N}}$ of i.i.d. random variables with the same law as $\bm{Z}$, and let $Z := \{\bm{Z}_i : 1 \leq i \leq n\}$. 
If $m$ of the variables in $Z$ were obtained by sampling from $\bar{V}$ (which arise whenever the Bernoulli random variable takes the value $B=1$), then without loss of generality 
 $Z | m = (\bm{Y}_1,\ldots, \bm{Y}_m, \bm{W}_{m+1},\ldots, \bm{W}_{n})$, where $\bm{W}_i \sim Q$.

From the law of conditional expectation, the fact that $X$ and $Z$ are identically distributed, the fact that $  h_{Z}^s | m \leq h_{Y^m}^s$ and \eqref{eqn: Reznikov},
\begin{eqnarray*}
\mathbb{E} [h_X^s] & = & \mathbb{E} \mathbb{E} [ h_Z^s | m] \\
& \leq & \mathbb{E} \mathbb{E} [ h_{Y^m}^s | m] \\
& = & \mathbb{E} \mathbb{E} [ \min(\text{diam}(M) , h_{Y^m}^s) \; | \; m] \\
& \leq & \mathbb{E} [ \min(\text{diam}(M) , C_s'' m^{-\frac{s}{d}} \log(m)^{\frac{s}{d}} ) ] .
\end{eqnarray*}
Fix $0 < \epsilon < w$.
From the law of total expectation, for $n > 1$:
\begin{eqnarray*}
& & \hspace{-60pt} \frac{ \mathbb{E} [ \min(\text{diam}(M) , C_s'' m^{-\frac{s}{d}} \log(m)^{\frac{s}{d}} ) ] }{n^{-\frac{s}{d}} \log(n)^{\frac{s}{d}} } \\
& = & \underbrace{ \text{diam}(M) \frac{ \mathbb{P}(m=0) }{n^{-\frac{s}{d}} \log(n)^{\frac{s}{d}} } }_{(*)} \\
& & + \underbrace{ \frac{ \mathbb{E} [ \min(\text{diam}(M) , C_s'' m^{-\frac{s}{d}} \log(m)^{\frac{s}{d}} ) | 0 < m \leq \epsilon n ] }{n^{-\frac{s}{d}} \log(n)^{\frac{s}{d}} } \mathbb{P}(0 < m \leq \epsilon n) }_{(**)} \\
& & + \underbrace{ \sum_{\tilde{m} > \epsilon n} \frac{\mathbb{E} [ \min(\text{diam}(M) , C_s'' \tilde{m}^{-\frac{s}{d}} \log(\tilde{m})^{\frac{s}{d}} ) | m = \tilde{m} ] }{n^{-\frac{s}{d}} \log(n)^{\frac{s}{d}} } \mathbb{P}(m = \tilde{m}) }_{(***)} 
\end{eqnarray*}
The first term is seen to vanish as $n \rightarrow \infty$:
\begin{eqnarray*}
(*) & = & \text{diam}(M) n^{\frac{s}{d}} \log(n)^{-\frac{s}{d}} (1-w)^n \rightarrow 0 
\end{eqnarray*}
Since $m \sim \text{Binomial}(n,w)$, we have that $\log(m) / \log(n) \leq 1$ and moreover from Hoeffding's inequality we have that
\begin{eqnarray*}
\mathbb{P}(m \leq \epsilon n) & \leq & \exp(-2(w-\epsilon)^2 n) .
\end{eqnarray*}
Thus, letting $C_s''' = \max(\text{diam}(M) , C_s'')$, the second term is also seen to vanish as $n \rightarrow \infty$:
\begin{eqnarray*}
(**) & \leq & \mathbb{E} \left[ \left. \min \left(\text{diam}(M) , C_s'' \left(\frac{m}{n}\right)^{-\frac{s}{d}} \right) \right| 0 < m \leq \epsilon n \right] \mathbb{P}(0 < m \leq \epsilon n) \\
& \leq & C_s''' n^{\frac{s}{d}} \mathbb{P}(0 < m \leq \epsilon n) \\
& \leq & C_s''' n^{\frac{s}{d}} \exp(-2(w-\epsilon)^2 n) \rightarrow 0 
\end{eqnarray*}
For the final term, let $g : [0,1] \rightarrow \mathbb{R}$ be defined as
\begin{eqnarray*}
g(x) & = & \left\{ \begin{array}{ll} \epsilon^{-\frac{s}{d}} & x \leq \epsilon \\ x^{-\frac{s}{d}} & x > \epsilon \end{array} \right. 
\end{eqnarray*}
which is observed to be a continuous and bounded.
From the strong law of large numbers, $m/n$ converges a.e. to $w$, and thus in distribution to $\delta(w)$. From the Portmanteau theorem we have that $\mathbb{E}[g(m / n)] \rightarrow g(w) = w^{-s/d}$.
Thus the third term is bounded as $n \rightarrow \infty$:
\begin{eqnarray*}
(***) & \leq & \sum_{\tilde{m} > \epsilon n} \min \left(\text{diam}(M) , C_s'' \left(\frac{\tilde{m}}{n}\right)^{-\frac{s}{d}} \right) \mathbb{P}(m = \tilde{m}) \\
& \leq & C_s''' \sum_{\tilde{m} > \epsilon n} \left(\frac{\tilde{m}}{n}\right)^{-\frac{s}{d}} \mathbb{P}(m = \tilde{m}) \\
& \leq & C_s''' \mathbb{E}\left[ g\left( \frac{m}{n} \right) \right] \; \rightarrow \; C_s''' w^{-\frac{s}{d}} \; < \; \infty
\end{eqnarray*}
Together with \Cref{thm: main result}, the result is now established.
\end{proof}

Now we turn to the proof of \Cref{cor: MCMC}:

\begin{proof}[Proof of \Cref{cor: MCMC}]
The proof focuses on $(\bm{x}_{in})_{i \in \mathbb{N}}$, which is sometimes referred to as the $n$-step \emph{jump chain} and $P^n$ is its transition kernel.
Since $M$ is compact and $P$ is continuous (\Cref{ass: density}), the $n$th step transition distributions initialised from $\bm{x}_{in} \in M$, denoted $P_{n,i}(\cdot) := P^n(\bm{x}_{in} , \cdot)$, are absolutely continuous with respect to $P$ and therefore admit Radon-Nikodym derivatives $\mathrm{d}P_{n,i} / \mathrm{d}P$. 
Since the Markov chain is ergodic, $| \mathrm{d}P_{n,i} / \mathrm{d}P - 1 | \leq C \rho^n$ and we can select a value $n_0 \in \mathbb{N}$ independent of the $i$ index such that, for some $w < 1$ and all $n \geq n_0$, it holds almost everywhere that $\mathrm{d}P_{n,i} / \mathrm{d}P \geq w > 0$.
It follows that $P_{n_0,i}$ can be represented as a bivariate mixture, one of whose components is $P$.
Specifically, $\bm{x}_{i n_0}$ is equal in distribution to $\bm{z}_i$, where
\begin{eqnarray*}
\bm{z}_i | \bm{z}_{i-1} , B_i & \sim & \left\{ \begin{array}{ll} P & \text{if } B_i = 1 \\ Q_i & \text{if } B_i = 0 \end{array} \right. , \qquad B_i \overset{\text{i.i.d.}}{\sim} \text{Bernoulli}(w), \qquad \frac{\mathrm{d}Q_i}{\mathrm{d}P} = \frac{1}{1-w}\left( \frac{\mathrm{d}P_{n_0,i-1}}{\mathrm{d}P} - w \right) .
\end{eqnarray*}
Consider then such a sequence $(\bm{z}_i)_{i \in \mathbb{N}}$ and let $Z = \{ \bm{z}_i : 1 \leq i \leq \lfloor n / n_0 \rfloor \}$ and $Y = \{ \bm{z}_i : B_i = 1, \; 1 \leq i \leq \lfloor n / n_0 \rfloor \}$.
The elements of $Y$ are independent samples from $P$ and $m := |Y| \sim \text{Binomial}(\lfloor n / n_0 \rfloor)$.
The remainder of the proof is identical to that used for \Cref{cor: MC}.
\end{proof}

\end{bibunit}

\newpage

\begin{bibunit}[plainnat]

\section*{Supplement}

In this electronic supplement we report experiments designed to validate our theoretical results for the Riemann--Stein kernel method.
For this purpose we considered arguably the most important compact manifold; the sphere $\mathbb{S}^2$.

\subsubsection*{Differential Operators}

The coordinate patch $\phi$ from \Cref{example: S2} in \Cref{sec: background} can be used to compute the metric tensor
$$
\mathrm{G} \; =\;  \left( \begin{array}{cc} \sin^2 q_2 & 0 \\ 0 & 1 \end{array} \right)
$$
and a natural volume element $\mathrm{d}V = \sin q_2 \; \mathrm{d}q_1 \mathrm{d} q_2$.
It follows that, for a function $h : \mathbb{S}^2 \rightarrow \mathbb{R}$, we have the gradient differential operator
\begin{eqnarray*}
\nabla h & = & \frac{1}{\sin^2 q_2} \frac{\partial h}{\partial q_1} \partial_{q_1} + \frac{\partial h}{\partial q_2} \partial_{q_2} .
\end{eqnarray*}
Similarly, for a vector field $\bm{s} = s_1 \partial_{q_1} + s_2 \partial_{q_2}$, we have the divergence operator
\begin{eqnarray*}
\nabla \cdot \bm{s} & = & \frac{\partial s_1}{\partial q_1} + \frac{\partial s_2}{\partial q_2} + \frac{\cos q_2}{\sin q_2} s_2 .
\end{eqnarray*}
Thus the Riemann--Stein operator $\tau$ is:
\begin{eqnarray}
\tau h & = & \frac{\cos q_2}{\sin q_2} \frac{\partial h}{\partial q_2} + \frac{1}{\sin^2 q_2} \left\{ \frac{1}{p} \frac{\partial p}{\partial q_1} \frac{\partial h}{\partial q_1} + \frac{\partial^2 h}{\partial q_1^2} \right\} + \left\{ \frac{1}{p} \frac{\partial p}{\partial q_2} \frac{\partial h}{\partial q_2} + \frac{\partial^2 h}{\partial q_2^2} \right\} . \label{eq: Lpi for S2}
\end{eqnarray}
Turning this into expressions in terms of $\bm{x}$ requires that we notice
$$
\frac{\cos q_2}{\sin q_2} \; = \; \frac{x_3}{\sqrt{1 - x_3^2}}, \qquad \frac{1}{\sin^2 q_2} \; =\; \frac{1}{1 - x_3^2}
$$
and use chain rule for partial differentiation.

\subsubsection*{Choice of Kernel} 

The performance of the Riemann--Stein kernel method depends, of course, on the selection of a reproducing kernel $k$ to define the space $\mathcal{H}(k)$.
For standard manifolds, such as the sphere $M = \mathbb{S}^2$, several function spaces and their reproducing kernels have been studied \citep[e.g.][]{Porcu2016}.
For more general manifolds, an extrinsic kernel can be induced under embedding into an ambient space \citep{Lin2017}, or the stochastic partial differential approach \citep{Fasshauer2011,Lindgren2011,Niu2017} can be used to numerically approximate a suitable intrinsic kernel.
Note that none of the theoretical development in this paper relies on an embedding of $M$ into an ambient space; all of our analysis is intrinsic to the manifold.

Although our focus is on $\mathbb{S}^2$, for generality the remainder of this section discusses kernels on $\mathbb{S}^d$.
A kernel on $\mathbb{S}^d$ is characterised by a scalar $\sigma > 0$ and a sequence $(b_{n,d})_{n=0}^\infty$ of $d$-\emph{Schoenberg coefficients}, such that $0 \leq b_{n,d}$ and $\sum_{n=0}^\infty b_{n,d} = 1$, in the sense that
\begin{eqnarray}
k(\bm{x},\bm{y}) & = & \sigma^2 \sum_{n=0}^\infty b_{n,d} \frac{C_n^{(d-1)/2}( \bm{x} \cdot \bm{y} )}{C_n^{(d-1)/2}(1) }, \qquad \bm{x}, \bm{y} \in \mathbb{S}^d
\end{eqnarray}
where $C_n^\lambda$ are the Geigenbauer polynomials of degree $n$ and order $\lambda > 0$ \citep{Bingham1973,Marinucci2011,Dai2013,Gneiting2013,Daley2013}.
It is known that $b_{n,d} \asymp n^{-2\alpha}$ if and only if $\mathcal{H}(k)$ is norm-equivalent to the Sobolev space $W_2^\alpha(\mathbb{S}^d)$ \citep{Daley2013}. 
Thus in principle one has much scope to design an appropriate kernel whilst also ensuring that \Cref{asm: Stein class} is satisfied.
Therefore we need to elicit some additional \textit{desiderata} (D) to constrain ourselves to those kernels which are most useful:
\begin{desiderata}
The kernel should have an explicit form that can be easily differentiated.
\end{desiderata}
\begin{desiderata}
The kernel should have easily customisable smoothness.
\end{desiderata}
For the Riemann--Stein kernel method to be practical, D1 must hold.
For the Stein kernel method to be flexibly used, D2 must hold.
Note that one cannot, for example, just restrict the Mat\'{e}rn kernel on $\mathbb{R}^3$ to $\mathbb{S}^2$; in order that the restriction is positive definite we require a strong condition $\nu \in (0,\frac{1}{2}]$ on the smoothness parameter \citep{Gneiting2013}.
This fact makes this approach not suitable for our work, where higher order derivatives of the kernel are needed.
Moreover, the restriction of a Sobolev space to a manifold incurs a loss-of-smoothenss that depends on the manifold.
In what follows, two different kernels that each reproduce Sobolev spaces are presented:

\begin{kernel}[Brauchart; \cite{Brauchart2013}] \label{ex: Brauchart kernel}
For $\alpha + \frac{1}{2} - \frac{d}{2} \in \mathbb{N}$ and $\alpha > \frac{d}{2}$, the kernel
\begin{eqnarray}
k_1(\bm{x},\bm{y})  & = & C^{(1)} {}_3 F_2\left[ \begin{array}{cc} \frac{d}{2} + \frac{1}{2} - \alpha, \frac{d}{2} - \alpha, \frac{d}{2} + \frac{1}{2} - \alpha \\ \frac{d}{2} + 1 - \alpha , 1 + \frac{d}{2} - 2\alpha \end{array} ; \frac{1 - \bm{x}\cdot\bm{y}}{2} \right] + C^{(2)}  \|\bm{x} - \bm{y}\|_2^{2\alpha-2}, \label{eq: S2 kernel}
\end{eqnarray}
defined for $\bm{x},\bm{y} \in \mathbb{S}^d$, reproduces the Sobolev space $W_2^\alpha(\mathbb{S}^d)$.
In this paper ${}_pF_q$ is  the generalised hypergeometric function.
The constant terms in the kernel in \eqref{eq: S2 kernel} are as follows:
\begin{eqnarray*}
C^{(1)} & = & \frac{2^{2\alpha - 2}}{2\alpha - d} \frac{(\frac{d}{2})_{2\alpha - 2}}{(d)_{2\alpha - 2}} \\
C^{(2)} & = & (-1)^{\alpha - \frac{d}{2} + \frac{1}{2}} 2^{d - 2\alpha - 1} \frac{\Gamma(\frac{d+1}{2}) \Gamma(\alpha - \frac{d}{2} + \frac{1}{2}) \Gamma(\alpha - \frac{d}{2} + \frac{1}{2})}{\sqrt{\pi} \Gamma(\frac{d}{2}) (\frac{1}{2})_{\alpha - \frac{d}{2} + \frac{1}{2}} (\frac{d}{2})_{\alpha - \frac{d}{2} + \frac{1}{2}} } 
\end{eqnarray*}
where $(z)_n := \Gamma(z+n) / \Gamma(z)$ is the Pochhammer symbol.
From properties of hypergeometric functions, this kernel has an explicit closed form when $\alpha + \frac{1}{2} - \frac{d}{2} \in \mathbb{N}$, so that D1 is satisfied.
Moreover, D2 holds for this kernel.
\end{kernel}

\begin{kernel}[Wendland; \cite{Wendland1995}] \label{ex: Wendland}
Consider the compact support positive definite functions due to \cite{Wendland1995}
\begin{eqnarray*}
\phi_{i,j}(r) & = & \left\{ \begin{array}{ll} p_{i,j}(r) & \text{if } r \leq 1 \\
0 & \text{if } r > 1 \end{array} \right. ,
\end{eqnarray*}
defined for $r \geq 0$ where $p_{i,j} : [0,\infty) \rightarrow \mathbb{R}$ is a particular polynomial selected such that $\phi_{i,j}$ is positive definite and $\phi_{i,j} \in C^{2j}$.
The radial basis function 
\begin{eqnarray*}
k_2(\bm{x},\bm{y}) & = & \phi_{i,j}\left(\frac{\|\bm{x} - \bm{y}\|_2}{\lambda}\right),
\end{eqnarray*}
for $\lambda > 0$ reproduces $W_2^{j + \frac{i}{2} + \frac{1}{2}}(\mathbb{R}^i)$ on $\bm{x},\bm{y} \in \mathbb{R}^i$ \citep[see e.g. Thm. 2.1 in][]{Wendland1998}.
In the particular case where $i = d+1$, the restriction of $k_2$ to $\bm{x},\bm{y} \in \mathbb{S}^d$ reproduces $W_2^\alpha(\mathbb{S}^d)$ with $\alpha = j + \frac{i}{2}$; see Thm. 4.1 of \cite{Narcowich2007}.
See also \cite{Gneiting2002a,Narcowich2002,Zastavnyi2006,Bevilacqua2017}.
Desiderata D2 is therefore satisfied.
Moreover, this kernel has a closed form so that D1 is satisfied.
\end{kernel}

\subsubsection*{Assessment}

To numerically assess the convergence of the Riemann--Stein kernel method, consider the von Mises-Fisher distribution $P$ whose density with respect to $V$ is
$$
p(\bm{x}) = \frac{\|\bm{c}\|_2}{4 \pi \; \text{sinh}(\|\bm{c}\|_2)} \exp(\bm{c}^\top \bm{x}).
$$
For illustration, we suppose that the normalisation constant is unknown and we are told only that $p(\bm{x}) \propto \exp( \bm{c}^\top \bm{x} )$.
This is sufficient to construct the differential operator $\tau$ as previously described.
Our aim in what follows is to validate our theoretical analysis; for this reason in all experiments we fixed
$\lambda = 2$ for $k_2$ as a convenient default.
Further theoretical work will be needed to understand the properties of the Riemann--Stein kernel method when kernel parameters are adaptively estimated \citep{Stein2012}.

In what follows we first considered point sets $X = \{\bm{x}_i\}_{i=1}^n$ whose elements were quasi-uniformly distributed on $\mathbb{S}^2$, being obtained by minimising a generalised electrostatic potential energy \citep[Reisz's energy;][]{Semechko2015}.
Note that these points, being uniform, do not arise as an approximation to $P$; rather, they are intended to asymptotically minimise $h_X$ as motivated by \Cref{thm: main result}.
First, explicit function approximations are presented based on $k_1$, $k_2$ respectively in \Cref{fig: function approx k1,fig: function approx k3}.
These function approximations are $\hat{f} = \hat{\xi} + \tau \hat{h}$ where $(\hat{\xi},\hat{h})$ solve \eqref{eq: Stein method}.
Here the integrand $f(x)$ was based on a Rosenbrock function and represents a modest challenge to an interpolation-based integration method.
It was observed that both kernels provided an accurate approximation when a large number of points were used.
Next, for various values of $n$, we computed the worst case integration error in \eqref{eq: wce} (i.e. the generalised KSD).
Results in \Cref{fig: converge 1,fig: converge 3} show that, for both kernels $k_1$ and $k_2$, convergence of the generalised KSD occurred at the rate that was theoretically predicted.
Finally, we re-evaluated the worst case integration error (i.e. the generalised KSD), based instead on a point set $X$ generated as the realisation of an MCMC sample path.
Here we employed a Metropolis-Hastings Markov chain with the normalised Riemannian measure $\bar{V}$ as the proposal.
Results are presented in \Cref{fig: converge mcmc 1,fig: converge mcmc 3}, where again the theoretically obtained convergence rate was validated.

\begin{figure}
\centering
\begin{subfigure}[b]{0.49\textwidth}
\includegraphics[width = \textwidth,clip,trim = 0cm 1cm 3.9cm 0cm]{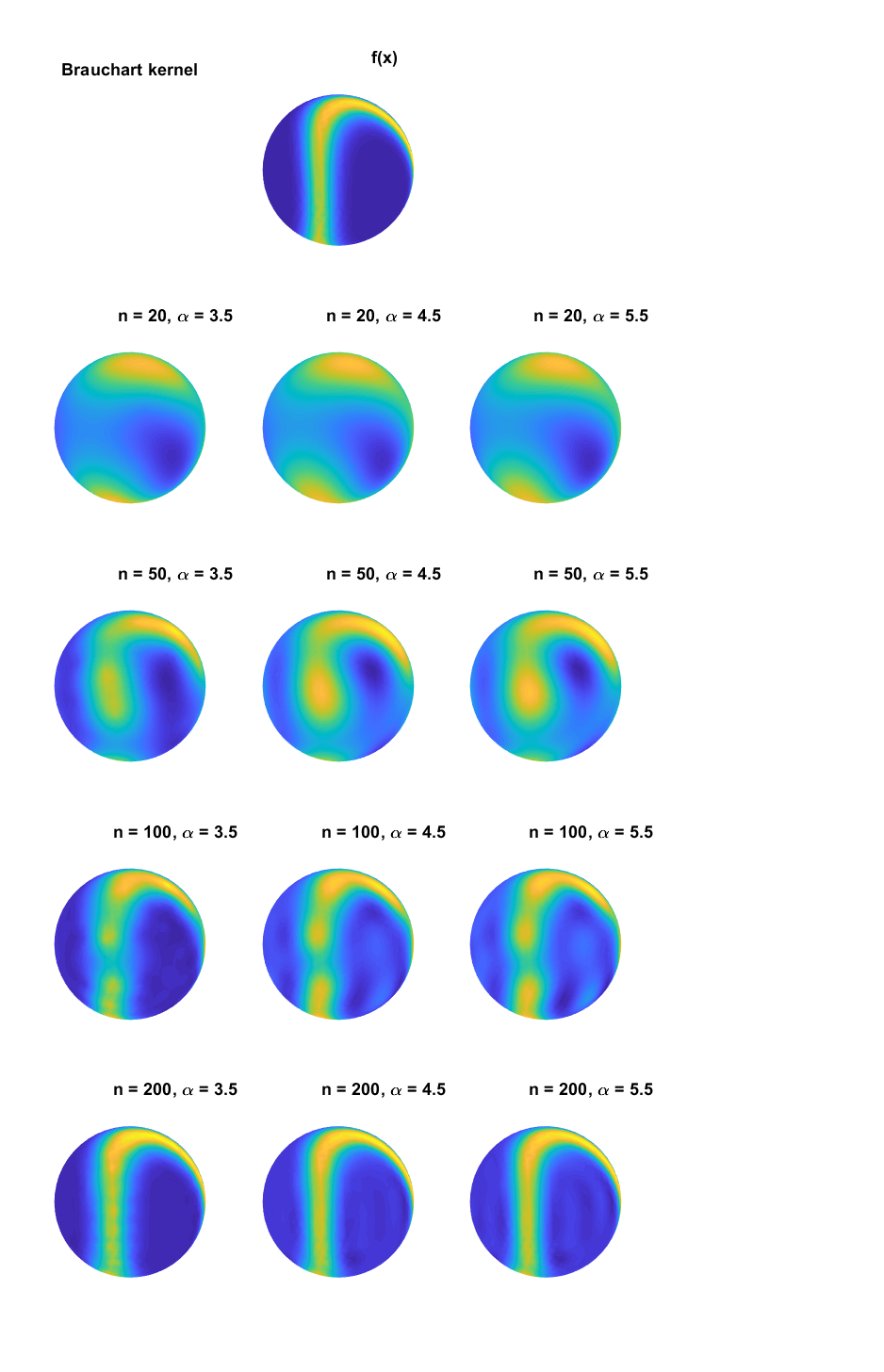}
\caption{$k_1$}
\label{fig: function approx k1}
\end{subfigure}
\begin{subfigure}[b]{0.49\textwidth}
\includegraphics[width = \textwidth,clip,trim = 0cm 1cm 3.9cm 0cm]{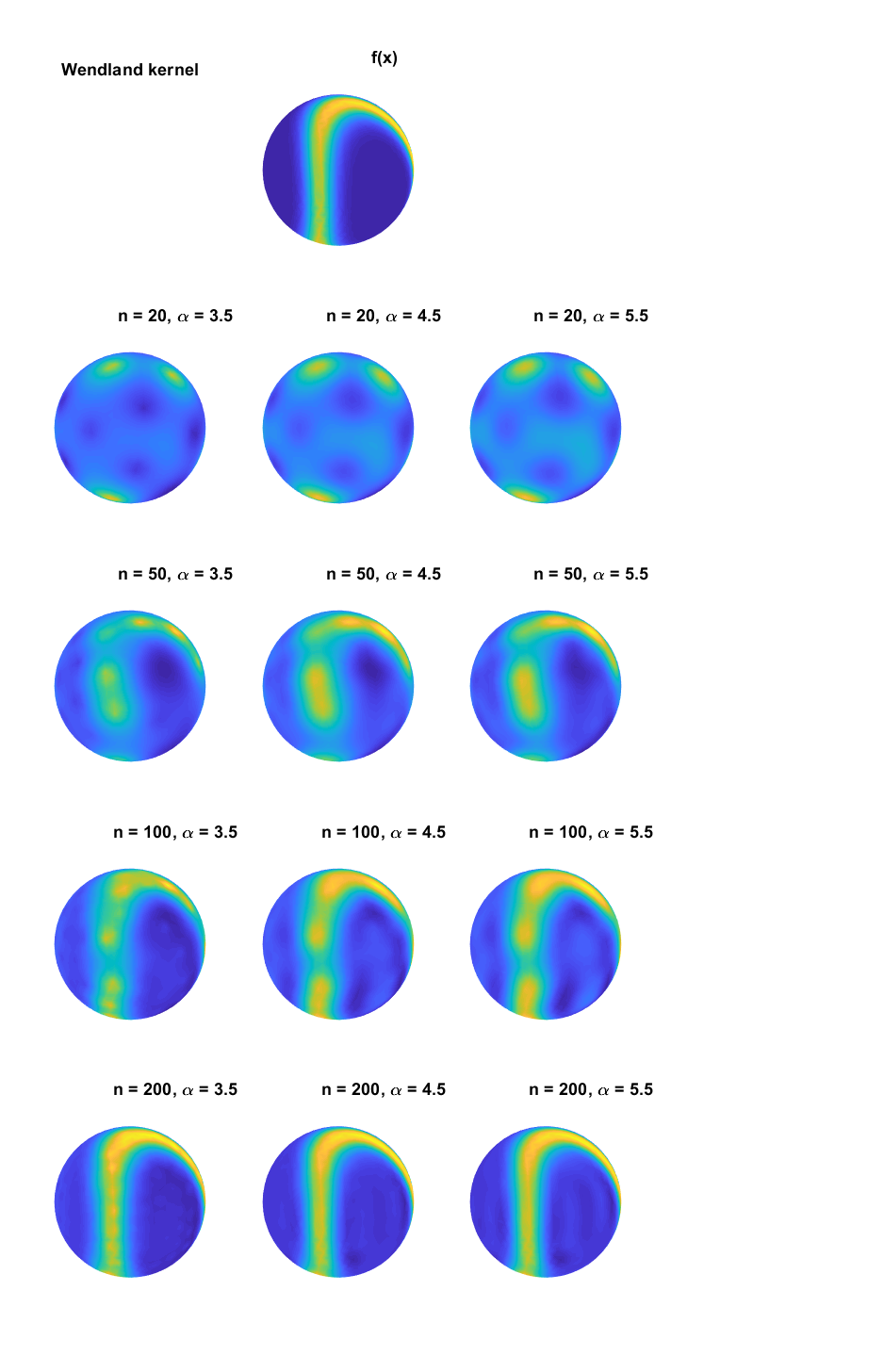}
\caption{$k_2$}
\label{fig: function approx k3}
\end{subfigure}
\caption{Function approximation with the Riemann--Stein kernel.
Here $f$ represents the exact integrand.
Each panel represents a kernel (a) $k_1$, (b) $k_2$, varying both the smoothness $\alpha$ of these kernel and the number $n$ of evaluations of the integrand.
The point set was quasi-uniform over $\mathbb{S}^2$.
}
\end{figure}

\begin{figure}
\centering
\begin{subfigure}[b]{0.4\textwidth}
\includegraphics[width = \textwidth,clip,trim = 5cm 9.5cm 5cm 9.5cm]{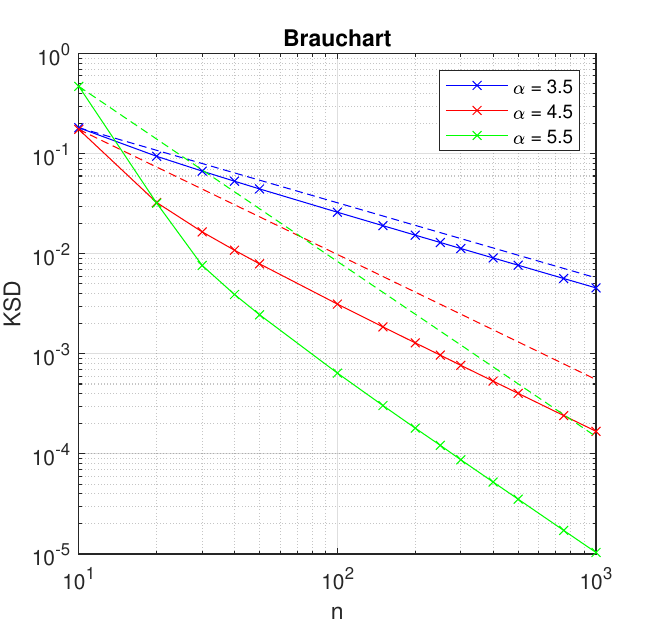} 
\caption{$k_1$}
\label{fig: converge 1}
\end{subfigure}
\begin{subfigure}[b]{0.4\textwidth}
\includegraphics[width = \textwidth,clip,trim = 5cm 9.5cm 5cm 9.5cm]{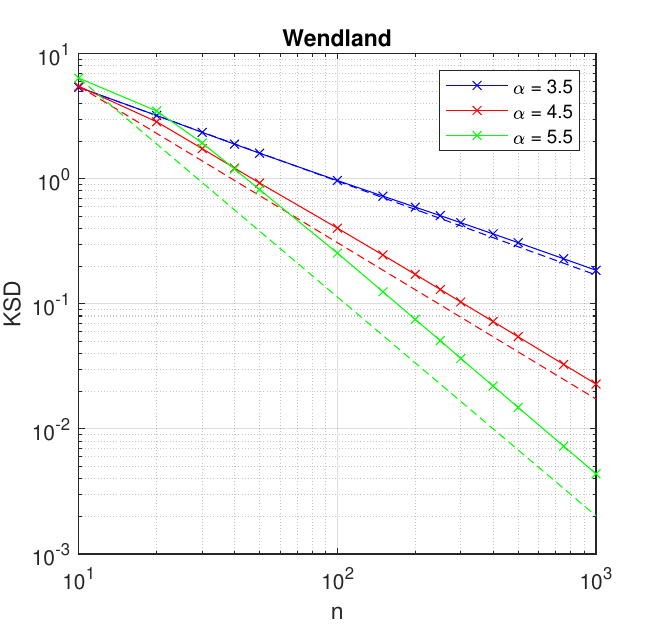}
\caption{$k_2$}
\label{fig: converge 3}
\end{subfigure}

\caption{
The worst case integration error (i.e. generalised kernel Stein discrepancy; KSD) of the Riemann--Stein kernel method was plotted for two different kernels, (a) $k_1$, (b) $k_2$, varying both the smoothness $\alpha$ of these kernels and the number $n$ of evaluations of the integrand.
The point set was quasi-uniform over $\mathbb{S}^2$.
Dashed lines represent the slope of the convergence rates that we have theoretically established.
}
\label{fig: converge}
\end{figure}

\begin{figure}
\centering
\begin{subfigure}[b]{0.4\textwidth}
\includegraphics[width = \textwidth,clip,trim = 5cm 9.5cm 5cm 9.5cm]{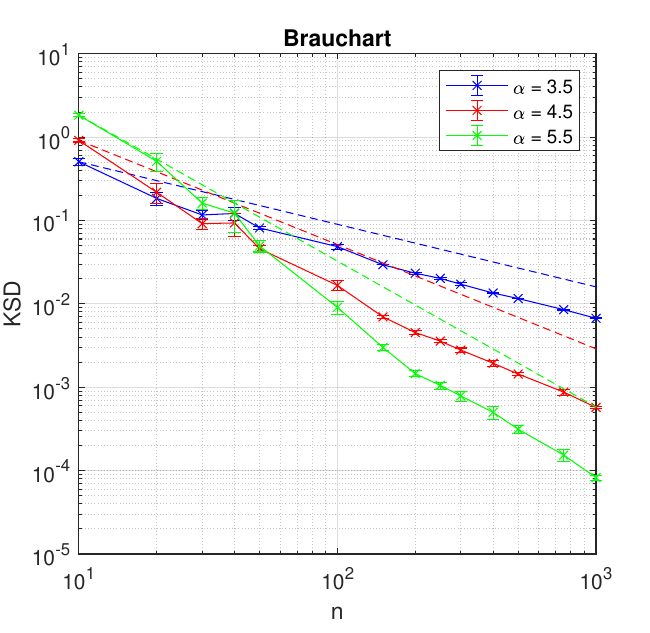} 
\caption{$k_1$}
\label{fig: converge mcmc 1}
\end{subfigure}
\begin{subfigure}[b]{0.4\textwidth}
\includegraphics[width = \textwidth,clip,trim = 5cm 9.5cm 5cm 9.5cm]{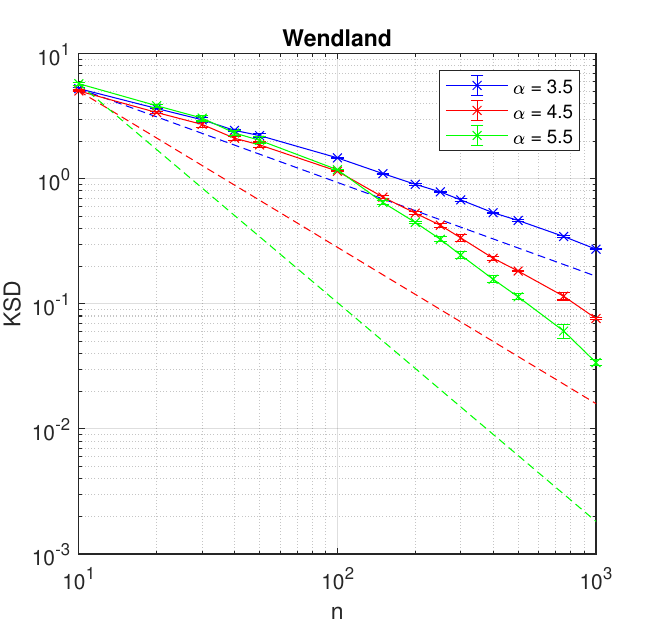}
\caption{$k_2$}
\label{fig: converge mcmc 3}
\end{subfigure}

\caption{
The worst case integration error (i.e. generalised kernel Stein discrepancy; KSD) of the Riemann--Stein kernel method was plotted for two different kernels, (a) $k_1$, (b) $k_2$, varying both the smoothness $\alpha$ of these kernels and the number $n$ of evaluations of the integrand.
The point set was obtained as the realisation of a Markov chain whose invariant distribution was $P$.
Here the arithmetic mean estimator is presented along with standard error bars, averaged over multiple realisations of the Markov chain.
Dashed lines represent the slope of the convergence rates that we have theoretically established.
}
\label{fig: converge mcmc}
\end{figure}

\end{bibunit}

\end{document}